\documentclass[preprint,12pt]{rfbelsarticle}

\usepackage{amsmath, amsthm, amssymb, amsfonts}
\usepackage{mathrsfs}
\usepackage{subfigure,multirow}
\usepackage{color}

\newtheorem{theorem}{Theorem}[subsection]
\newtheorem{lemma}[theorem]{Lemma}
\newtheorem{corollary}[theorem]{Corollary}
\newtheorem{proposition}[theorem]{Proposition}

\theoremstyle{definition}

\newtheorem{definition}[theorem]{Definition}
\newtheorem{example}[theorem]{Example}
\newtheorem*{problem}{Problem}

\newtheorem{appexample}{Example}[section]

\renewcommand{\AA}{\mathbf{A}}
\newcommand{\BB}{\mathbf{B}}
\newcommand{\XX}{\mathbf{X}}
\newcommand{\TT}{\mathbf{T}}

\newcommand{\vv}{\mathbf{v}}
\newcommand{\kk}{\mathbf{k}}
\renewcommand{\tt}{\mathbf{t}}
\newcommand{\xx}{\mathbf{x}}
\newcommand{\yy}{\mathbf{y}}
\newcommand{\kkappa}{\boldsymbol{\kappa}}
\newcommand{\sts}{\mathrm{STS}}

\newcommand{\kts}{\mathrm{KTS}}
\newcommand{\nkts}{\mathrm{NKTS}}

\newcommand{\skss}{\mathrm{SKSS}}

\newcommand{\PA}{\mathrm{PA}}
\newcommand{\rs}{\mathrm{RS}}
\newcommand{\soma}{\mathrm{SOMA}}

\newcommand{\nequiv}{\not\equiv}

\newcommand{\tv}{\widetilde{v}}
\newcommand{\tx}{\widetilde{X}}

\newcommand{\blob}{\circle*{0.2}}

\begin{document}

\begin{frontmatter}

\title{Generalized packing designs}
\author[rfb]{Robert F. Bailey}
\ead{robert.bailey@uregina.ca}

\author[acb]{Andrea C. Burgess\corref{cor1}}
\ead{andrea.burgess@ryerson.ca}

\cortext[cor1]{Corresponding author}

\address[rfb]{Department of Mathematics and Statistics, University of Regina, 3737 Wascana Parkway, Regina, Saskatchewan S4S~0A2, Canada}
\address[acb]{Department of Mathematics, Ryerson University, 350 Victoria St., Toronto, Ontario, M5B~2K3, Canada}



\begin{abstract}
Generalized $t$-designs, which form a common generalization of objects such as $t$-designs, resolvable designs and orthogonal arrays, were defined by Cameron [P.J.~Cameron, A generalisation of $t$-designs, \emph{Discrete Math.}\ 309 (2009), 4835--4842].  In this paper, we define a related class of combinatorial designs which simultaneously generalize packing designs and packing arrays.  We describe the sometimes surprising connections which these generalized designs have with various known classes of combinatorial designs, including Howell designs, partial Latin squares and several classes of triple systems, and also concepts such as resolvability and block colouring of ordinary designs and packings, and orthogonal resolutions and colourings.  Moreover, we derive bounds on the size of a generalized packing design and construct optimal generalized packings in certain cases.  In particular, we provide methods for constructing maximum generalized packings with $t=2$ and block size $k=3$ or 4.
\end{abstract}


\end{frontmatter}

\section{Introduction}
In his 2009 paper~\cite{cameron}, Cameron introduced a new class of combinatorial designs, which simultaneously generalizes various well-known classes of designs, including $t$-designs, mutually orthogonal Latin squares, orthogonal arrays and 1-factorizations of complete graphs.  Further work on Cameron's ``generalized $t$-designs'' has been done by Soicher~\cite{soicher} and others~\cite{Drizen,Patterson}.  Related objects are also discussed in the earlier papers of Martin~\cite{Martin98,Martin99} and Teirlinck~\cite{Teirlinck89}.  In a remark near the end of his paper, Cameron suggests that a similar definition can be made for generalizing both packing and covering designs.

In a recent paper~\cite{gencov}, the authors, Cavers and Meagher considered the analogue of Cameron's generalization for covering designs.  In this paper, we pursue the ``dual'' notion of {\em generalized packing designs}.  The key difference when studying packing or covering problems rather than ``traditional'' designs is that the question is typically not whether the designs exist (this is usually trivial to answer), but obtaining bounds on the maximum (for packings) or minimum (for coverings) size, and constructing optimal (or near-optimal) designs.  However, the similarity between packing and covering only goes so far.  In this paper, we shall see how a number of families of known designs (including Howell designs, partial Latin squares and several classes of triple systems) arise as special cases of generalized packing designs.  We shall also see how concepts such as resolvability and block colouring of ordinary designs and packings, and orthogonal resolutions and colourings, appear in this setting.

Background material on most classes of designs can be found in the {\em Handbook of Combinatorial Designs}~\cite{handbook}.  Before introducing our generalized packing designs, we will review ordinary packing designs.

\subsection{Ordinary packing designs}

\begin{definition} \label{defn:packing}
Let $v,k,t,\lambda$ be positive integers with $v\geq k\geq t$.  A {\em $t$-$(v,k,\lambda)$ packing design}, or more succinctly a {\em packing}, is a family $\mathcal{D}$ of $k$-subsets (called {\em blocks}) of a $v$-set $X$, where any $t$-subset of $X$ is contained in at most $\lambda$ members of $\mathcal{D}$.
\end{definition}

\begin{example}
The following is an example of a $2$-$(6,3,1)$ packing:
$$
\begin{array}{l}
\{ 1,2,4 \} \\
\{ 2,3,5 \} \\
\{ 3,4,6 \} \\
\{ 1,5,6 \}.
\end{array}
$$
It is straightforward to check that no 2-subset of $\{1,\ldots,6\}$ appears in more than one block.  Also, this packing is of maximum possible size.
\end{example}

Note that in the case where each $t$-subset occurs in {\em exactly} $\lambda$ blocks, we have a $t$-$(v,k,\lambda)$ design.  However, as we have a weaker requirement, it is trivial to show that a $t$-$(v,k,\lambda)$ packing exists: a single $k$-subset satisfies the definition almost vacuously.  Instead, what is considered interesting is to determine the maximum possible size of a $t$-$(v,k,\lambda)$ packing, and give constructions of packings which meet that bound.  To that end, we make the following definition.

\begin{definition} \label{defn:packing_number}
Let $v,k,t,\lambda$ be positive integers with $v\geq k\geq t$.  The {\em packing number} $D_\lambda(v,k,t)$ is the maximum possible number of blocks in a $t$-$(v,k,\lambda)$ packing.
\end{definition}

In this paper, we are primarily interested in the case where $\lambda=1$, in which case we omit the subscript~$\lambda$.

There are a number of bounds known on $D_\lambda(v,k,t)$; the reader is referred to the survey by Mills and Mullin~\cite{mills_mullin} for details.  (An updated list of results on packing numbers can be found in~\cite[{\S}VI.40]{handbook}.)  The most general bound was found independently in the 1960s by both Johnson~\cite{johnson} and Sch\"onheim~\cite{schonheim_cover,schonheim} (in fact, Johnson was studying the equivalent problem of bounding the size of a constant-weight binary error-correcting code).  We shall refer to this bound as the {\em Johnson--Sch\"onheim bound}.

\begin{proposition} \label{johnson_schonheim}
Let $v,k,t,\lambda$ be positive integers with $v\geq k\geq t$.  Then the packing number satisfies
\[ D_\lambda(v,k,t) \leq U_\lambda(v,k,t) = \left\lfloor \frac{v}{k} \left\lfloor \frac{v-1}{k-1} \cdots \left\lfloor \frac{\lambda(v-t+1)}{k-t+1} \right\rfloor \cdots \right\rfloor \right\rfloor. \]
\end{proposition}

Various refinements of this bound are known.  For many small values of $k$, $t$ and $\lambda$, the structure of maximum $t$-$(v,k,\lambda)$ packings is completely described: the cases $t=2$, $\lambda=1$ and $k=3$ or~4 will be especially important in this paper, and these will be discussed later in Sections~\ref{section:blocksize3} and~\ref{section:blocksize4}, respectively.

\section{Generalized packings}

\subsection{Definitions and notation}
To define our generalized packing designs, we require various pieces of notation and terminology.  

If $\xx=(x_1,x_2,\ldots,x_m)$ and $\yy=(y_1,y_2,\ldots,y_m)$ are $m$-tuples of integers, we write $\xx\leq\yy$ to mean that $x_i\leq y_i$ for all $i\in\{1,2,\ldots,m\}$.  Similarly, if $\AA=(A_1,A_2,\ldots,A_m)$ and $\BB=(B_1,B_2,\ldots,B_m)$ are $m$-tuples of sets, we write $\AA\subseteq\BB$ to mean that $A_i\subseteq B_i$ for all $i\in\{1,2,\ldots,m\}$, and say {\em $\AA$ is contained in $\BB$}.

For any set $X$, we use the notation ${X \choose k}$ to denote the set of all $k$-subsets of $X$.  (Thus if $X$ is finite and has size $n$, then the size of ${X \choose k}$ is ${n \choose k}$.)  If we have an $m$-tuple of sets $\XX=(X_1,X_2,\ldots,X_m)$ and an $m$-tuple of integers $\kk=(k_1,k_2,\ldots,k_m)$, define
\[ {\XX \choose \kk} = {X_1 \choose k_1} \times {X_2 \choose k_2} \times \cdots \times {X_m \choose k_m}. \]
So a member of ${\XX \choose \kk}$ consists of an $m$-tuple of finite sets, of sizes $(k_1,k_2,\ldots,k_m)$.

Now suppose $v,k,t,\lambda$ are integers where $v \geq k \geq t \geq 1$ and $\lambda \geq 1$.  Let $\vv = (v_1,v_2,\ldots,v_m)$ be an $m$-tuple of positive integers with sum $v$, and let $\kk=(k_1,k_2,\ldots,k_m)$ be an $m$-tuple of positive integers with sum $k$, and where $\kk \leq \vv$.
Then let $\XX = (X_1,X_2,\ldots,X_m)$ be an $m$-tuple of pairwise disjoint sets, where $|X_i|=v_i$.  Let $\tt=(t_1,t_2,\ldots,t_m)$ be an $m$-tuple of {\em non-negative} integers.  We say $\tt$ is {\em $(\kk,t)$-admissible} if $\tt \leq \kk$ and $\sum t_i=t$.  In a similar manner, if $\TT=(T_1,T_2,\ldots,T_m)$ is an $m$-tuple of disjoint sets, we say that $\TT$ is {\em $(\vv,\kk,t)$-admissible} if each $T_i$ is a $t_i$-subset of $X_i$, where $(t_1,t_2,\ldots,t_m)$ is $(\kk,t)$-admissible.  (Note that since $t_i$ is allowed to be zero, the corresponding set $T_i$ is allowed to be empty.)

\begin{definition} \label{defn:gen_pack}
Suppose $\vv,\kk,t,\lambda,\XX$ are as above.  Then a {\em $t$-$(\vv,\kk,\lambda)$ generalized packing design}, or more succinctly a {\em generalized packing}, is a family $\mathcal{P}$ of elements of ${\XX \choose \kk}$, called {\em blocks}, with the property that every $\TT=(T_1,T_2,\ldots,T_m)$ which is $(\vv,\kk,t)$-admissible is contained in at most $\lambda$ blocks in $\mathcal{P}$.
\end{definition}

We call $X = X_1 \dot\cup X_2 \dot\cup \cdots \dot\cup X_m$ the {\em point set} of the generalized packing design; one can think of $\XX$ as being a partition of the point set $X$.  However, by an abuse of notation, we will often label the elements of each $X_i$ as $\{1,2,\ldots,v_i\}$.

We remark that our definition of a generalized packing is identical to Cameron's definition of a generalized $t$-design, except his definition requires ``exactly $\lambda$''.  It is also identical to that given in~\cite{gencov} for generalized covering designs, except that definition requires ``at least $\lambda$''.  Clearly, a generalized $t$-design is simultaneously a generalized packing and a generalized covering design.

As with ordinary packings, the existence of a $t$-$(\vv,\kk,\lambda)$ generalized packing is trivial to establish: a single block satisfies the definition.  So the interesting question is to bound the size of a generalized packing.  Again borrowing the notation from ordinary packings, we make the following definition.

\begin{definition} \label{defn:generalized packing_number}
Suppose $\vv,\kk,t,\lambda,\XX$ are as above.  The {\em generalized packing number} $D_\lambda(\vv,\kk,t)$ is the maximum possible number of blocks in a $t$-$(\vv,\kk,\lambda)$ generalized packing.
\end{definition}

Again, we are usually only interested in the case where $\lambda=1$, in which case we omit the subscript~$\lambda$.  Various bounds on $D(\vv,\kk,t)$ are given in Section~\ref{section:bounds}.  Before we do so, we shall consider some straightforward examples.

\subsection{Basic examples}

That we do indeed have a generalization of ordinary packings is shown by the next result.

\begin{proposition} \label{prop:ordinary}
Suppose $\vv=(v)$ and $\kk=(k)$.  Then a $t$-$(\vv,\kk,\lambda)$ generalized packing is equivalent to an ordinary $t$-$(v,k,\lambda)$ packing.
\end{proposition}

However, numerous other objects arise as generalized packings, as we spend much of this paper demonstrating.  An easy example is the following.

\begin{proposition} \label{prop:edgecolouring}
Suppose $\vv=(v_1,v_2)$, $\kk=(2,1)$, $t=2$ and $\lambda=1$.  Then a $2$-$(\vv,\kk,1)$ generalized packing is equivalent to a proper edge colouring of a simple graph on $v_1$ vertices, using at most $v_2$ colours.
\end{proposition}

\begin{proof}
Suppose we have such a graph.  An edge $\{x,y\}$ with colour $\alpha$ corresponds to a block $(\{x,y\},\, \{\alpha\})$.  The two admissible vectors $\tt$ are $\tt=(2,0)$ and $\tt=(1,1)$.  That no $\TT$ corresponding to $\tt=(2,0)$ is repeated is because the graph is simple; that no $\TT$ corresponding to $\tt=(1,1)$ is repeated is saying that no colour can appear more than once at a vertex, i.e.\ the colouring of the edges is proper.

On the other hand, given such a generalized packing, we can always construct an edge-coloured graph from it.
\end{proof}

\begin{example} \label{example:edgecolouring}
Suppose $\vv=(5,4)$ and $\kk=(2,1)$.  Figure~\ref{figure:edgecolouring} shows $2$-$(\vv,\kk,1)$ packing equivalent to the 4-edge-colouring of the given graph on 5 vertices.
\begin{figure}[htb]
\centering
\begin{minipage}[c]{50mm}
\[ \begin{array}{l}
( \{1,2\},\, \{a\} ) \\
( \{1,4\},\, \{b\} ) \\
( \{1,5\},\, \{c\} ) \\
( \{2,3\},\, \{b\} ) \\
( \{2,4\},\, \{c\} ) \\
( \{3,4\},\, \{d\} ) \\
( \{4,5\},\, \{a\} ) \\
\end{array} \]
\end{minipage}
\hspace{10mm}
\begin{minipage}[c]{50mm}
\setlength{\unitlength}{0.8cm}
\begin{picture}(4,5.5)
\thicklines

\put(0.5,0.5){\blob}
\put(3.5,0.5){\blob}
\put(0.5,3.5){\blob}
\put(3.5,3.5){\blob}
\put(2.0,5.0){\blob}

\put(0.5,0.5){\line(1,0){3}}
\put(0.5,0.5){\line(0,1){3}}
\put(0.5,0.5){\line(1,1){3}}
\put(3.5,0.5){\line(0,1){3}}
\put(3.5,3.5){\line(-1,0){3}}
\put(0.5,3.5){\line(1,1){1.5}}
\put(3.5,3.5){\line(-1,1){1.5}}

\put(0.0,3.3){1}
\put(0.0,0.3){2}
\put(3.75,0.3){3}
\put(3.75,3.3){4}
\put(1.85,5.25){5}

\put(0.1,1.9){$a$}
\put(1.85,0.0){$b$}
\put(3.65,1.9){$d$}
\put(1.85,3.0){$b$}
\put(1.85,2.2){$c$}
\put(0.9,4.25){$c$}
\put(2.9,4.25){$a$}

\end{picture}
\end{minipage}
\caption{A $2$-$(\vv,\kk,1)$ generalized packing, for $\vv=(5,4)$ and $\kk=(2,1)$.}
\label{figure:edgecolouring}
\end{figure}
\end{example}

\section{General results}
Throughout the remainder of the paper, unless otherwise specified, we let $\vv=(v_1, v_2, \ldots, v_m)$ and $\kk= (k_1, k_2, \ldots, k_m)$ and assume that $\vv \geq \kk$.  

\subsection{A few bounds} \label{section:bounds}

As mentioned above, one of our goals is to determine the maximum size of a given generalized packing.  In this subsection, we obtain a number of upper bounds on the generalized packing number $D_\lambda(\vv,\kk,t)$, particularly when $\lambda=1$.  Many of the results are analogous to lower bounds on the sizes of generalized covering designs given in~\cite{gencov}.  In many cases, the proofs are sufficiently similar to those in~\cite{gencov} that we refer the reader there for full details.

We begin by giving a bound based on the ordinary packing number, which is similar to~\cite[Corollary 3.10]{gencov}.  

\begin{proposition} \label{indiv_part}
Suppose $\vv=(v_1,v_2,\ldots,v_m)$ and $\kk=(k_1,k_2,\ldots,k_m)$ where $\vv\geq\kk$.
Then $\displaystyle D_{\lambda}(\vv,\kk,t) \leq \min_{k_i \geq t} D_{\lambda}(v_i, k_i, t)$.
\end{proposition}

\begin{proof}
The vector $\tt$ which has $t$ in position $i$ and $0$ elsewhere is admissible whenever $k_i \geq t$.  Now, the entries of the $i^{\textnormal{th}}$ component of each block form a $t$-$(v_i,k_i,\lambda)$ ordinary packing, and the result follows.
\end{proof}

If the bound given in Proposition~\ref{indiv_part} is met with equality, with $D_{\lambda}(\vv,\kk,t) = D_{\lambda}(v_i, k_i, t)$, then increasing the size of any part other than the $i^{\textnormal{th}}$ does not change the packing number.  We formalize this idea, which will prove crucial in determining the packing number in many cases, as follows.

\begin{proposition} \label{increase_v}
Suppose $\vv=(v_1, v_2, \ldots, v_m)$ and $\kk = (k_1, k_2, \ldots, k_m)$ and that there exists an $i \in \{1,2,\ldots,m\}$ such that $D_\lambda(\vv,\kk,t) = D_\lambda(v_i, k_i, t)$.  For $j \neq i$, suppose that $v_j' \geq v_j$.  Let $\vv' = (v_1', v_2', \ldots, v_{i-1}', v_i, v_{i+1}', v_{i+2}', \ldots, v_{m}')$.  Then $D_\lambda(\vv',\kk,t) = D_\lambda(v_i, k_i, t)$.
\end{proposition}

\begin{proof}
Let $\mathcal{P}$ be a maximum $t$-$(\vv,\kk,\lambda)$ packing.  Then the blocks of $\mathcal{P}$ form a $t$-$(\vv',\kk',\lambda)$ packing, where $v_j'-v_j$ points in $X_j$ are unused (whenever $j \neq i$).  By Proposition~\ref{indiv_part}, $D_\lambda(\vv',\kk',t)$ cannot exceed the size of $\mathcal{P}$.
\end{proof}

By considering an admissible vector $\tt$ with $t$ entries equal to~1 and all other entries~0, we obtain our next bound, which is somewhat reminiscent of the Johnson--Sch\"onheim bound (Proposition~\ref{johnson_schonheim}).  It is analogous to~\cite[Proposition 5.1]{gencov} for generalized covering designs.

\begin{lemma} \label{induction_bound}
Let $\vv=(v_1, v_2, \ldots, v_m)$, $\kk = (k_1, k_2, \ldots, k_m)$ and $\XX = (X_1, X_2, \ldots, X_m)$ be defined as above, and suppose that $t \leq m$.  Let $\{i_1, \ldots, i_t\}$ be a $t$-subset of $\{1, \ldots, m\}$, and let $\mathscr{B}$ be a collection of blocks with the property that each $t$-tuple of the form $(x_{i_1}, x_{i_2}, \ldots, x_{i_t})$, where $x_{i_j} \in X_{i_j}$, appears in at most one block.  Then 
$$
|\mathscr{B}| \leq \left\lfloor \frac{v_{i_1}}{k_{i_1}} \left \lfloor \frac{v_{i_2}}{k_{i_2}} \cdots \left\lfloor \frac{v_{i_t}}{k_{i_t}} \right\rfloor \cdots \right\rfloor \right\rfloor.
$$
\end{lemma}

By considering all possible such admissible vectors $\tt$, we have the following corollary (analogous to~\cite[Corollary 5.2]{gencov} for generalized covering designs).

\begin{corollary} \label{schonheim_type}
Let $\vv=(v_1, v_2, \ldots, v_m)$ and $\kk = (k_1, k_2, \ldots, k_m)$ and suppose that $t \leq m$.  Let $\mathscr{I}$ denote the collection of all $t$-subsets of $\{1, 2, \ldots, m\}$.  Then
$$
D(\vv,\kk,t) \leq \min_{\{i_1, \ldots, i_t\} \in \mathscr{I}} \left\lfloor \frac{v_{i_1}}{k_{i_1}} \left \lfloor \frac{v_{i_2}}{k_{i_2}} \cdots \left\lfloor \frac{v_{i_t}}{k_{i_t}} \right\rfloor \cdots \right\rfloor \right\rfloor.
$$
\end{corollary}


In the particular case that $t=2$ and $\lambda=1$, by combining the results of Proposition~\ref{indiv_part} and Corollary~\ref{schonheim_type}, we obtain the following bound.

\begin{proposition} \label{t=2_bound}
Let $\vv=(v_1, v_2, \ldots, v_m)$ and $\kk=(k_1, k_2, \ldots, k_m)$, where $\vv\geq \kk$ and $m \geq 2$.  Then
$$
D(\vv,\kk,2) \leq \min\left\{ \min_{k_i \geq 2} D(v_i, k_i, 2), \min_{\substack{i,j \in \{1, \ldots, m\} \\ i \neq j}} \left\lfloor \frac{v_i}{k_i} \left\lfloor \frac{v_j}{k_j}\right\rfloor\right\rfloor \right\}
$$
\end{proposition}

We conclude this section by providing a way of constructing a generalized packing from an existing one by merging parts.  Again, this is an analogy of an idea for generalized covering designs (see~\cite[Proposition 3.22]{gencov}).

\begin{proposition} \label{merging}
Let $\vv=(v_1, v_2, \ldots, v_m)$ and $\kk=(k_1, k_2, \ldots, k_m)$, and suppose there exists a $t$-$(\vv,\kk,\lambda)$ packing with $N$ blocks.  Then for all $i,j\in \{1, 2, \ldots, m\}$ with $i < j$, there exists a $t$-$(\vv^+,\kk^+,\lambda)$ packing with $N$ blocks, where 
$$\vv^+=(v_1, \ldots, v_{i-1}, v_{i+1}, \ldots, v_{j-1}, v_{j+1}, \ldots, v_m,v_i+v_j)$$
and 
$$\kk^+=(k_1, \ldots, k_{i-1}, k_{i+1}, \ldots, k_{j-1}, k_{j+1}, \ldots, k_m,k_i+k_j).$$
In particular, $D_\lambda(\vv^+,\kk^+,t) \geq D_\lambda(\vv,\kk,t)$.
\end{proposition}

\begin{proof}
Let $\mathscr{B}$ denote the collection of blocks in the $t$-$(\vv,\kk,\lambda)$ design.  We form a new collection of blocks $\mathscr{B}^+$ in the following way.  For each block $(B_1, B_2, \ldots, B_m) \in \mathscr{B}$, let 
$$(B_1, \ldots, B_{i-1}, B_{i+1}, \ldots, B_{j-1}, B_{j+1}, \ldots, B_m, B_i \cup B_j) \in \mathscr{B}^+.$$
It is easy to see that the $N$ blocks in $\mathscr{B}'$ form a $t$-$(\vv',\kk',\lambda)$ packing, and thus the bound follows.
\end{proof}

\subsection{The case $t=2$ and $\lambda=1$: a graphical interpretation} \label{section:graphical}

In~\cite{gencov}, many of the results obtained for generalized covering designs made use of an interpretation in terms of graphs.  Such an interpretation is also available for generalized packings.

Suppose $G$ is a graph, and $H$ a subgraph of $G$.  An {\em $H$-packing} of $G$ is a collection of edge-disjoint subgraphs of $G$, each isomorphic to $H$.  Now, an ordinary $2$-$(v,k,1)$ packing design can easily be regarded as a $K_k$-packing of $K_v$: that the subgraphs are edge-disjoint is equivalent to the condition that no pair of points occurs in more than one block.  We can also represent generalized packings in terms of graphs: to do so requires the following definition.

\begin{definition} \label{defn:join}
Let $G_1=(V_1,E_1)$ and $G_2=(V_2,E_2)$ be graphs with $V_1\cap V_2=\emptyset$.  Then the {\em join} of $G_1$ and $G_2$, denoted $G_1+G_2$, is the graph with vertex set $V_1\cup V_2$, and whose edge set is $E_1 \cup E_2 \cup \{xy\, :\, x\in V_1, y\in V_2\}$.
\end{definition}

For example, the join of two complete graphs is also complete, and the join of two empty graphs is a complete bipartite graph.  We note that this can be extended to a join of any number of graphs, and that this operation is associative.

Now suppose that $\vv=(v_1,v_2,\ldots,v_m)$  and $\kk=(k_1,k_2,\ldots,k_m)$ are vectors of positive integers with $\kk\leq\vv$.  We define a graph as follows,
\[ H_i=\left\{\begin{array}{cl}
\overline{K_{v_i}}, & \mbox{if $k_i=1$,}\\
K_{v_i}, & \mbox{if $k_i\geq 2$,}\\
\end{array}\right. \]
where $\overline{K_{v_i}}$ represents the complement of $K_{v_i}$ (that is, the empty graph).
Form the graph 
\[ G_{\vv,\kk} = H_1+\cdots+H_m \]
consisting of the join of the graphs $H_i$ such that $G_{\vv,\kk}$ has vertex set $V=\bigcup_i X_i$, where $|X_i|=v_i$ and each $X_i$ is the set of vertices of the corresponding $H_i$.

Analagous to~\cite[Theorem 3.5]{gencov} for generalized covering designs, we have the following result.

\begin{theorem} \label{thm:graph_packing}
Let $G_{\vv,\kk}$ be the graph described above.  Then a $2$-$(\vv,\kk,1)$ generalized packing is equivalent to a $K_k$-packing of $G_{\vv,\kk}$, with the property that for each copy of $K_k$, there are $k_i$ vertices in the set $X_i$ (for each $i$).
\end{theorem}



In the case of ordinary packings, where $\vv=(v)$ and $\kk=(k)$, this interpretation reduces to packing copies of $K_k$ into a complete graph $K_v$, a common way of thinking about packings.  In this situation, the {\em leave graph} (or the {\em leave} for short) is defined to be the subgraph of $K_v$ obtained by deleting the edges from all the blocks.  We give an analoguous definition for generalized packings below.

\begin{definition} \label{defn:leave}
Let $\mathcal{P}$ be a $2$-$(\vv,\kk,1)$ generalized packing.  The {\em leave graph}, or {\em leave}, of $\mathcal{P}$ is the subgraph of $G_{\vv,\kk}$ obtained by deleting the edges contained in blocks of $\mathcal{P}$.
\end{definition}

\begin{example} \label{example:K3packing}
Recall Example~\ref{example:edgecolouring}, where we exhibited a $2$-$(\vv,\kk,1)$ packing with $\vv=(5,4)$ and $\kk=(2,1)$.  By Theorem~\ref{thm:graph_packing}, this packing may be viewed as a $K_3$-packing of the graph $G_{\vv,\kk}$, which is illustrated, along with its leave, in Figure~\ref{figure:leave}.
\begin{figure}[htbp]
\centering
\subfigure[The graph $G_{\vv,\kk}$, where $\vv=(5,4)$ and $\kk=(2,1)$.]{
\includegraphics[width=0.3\textwidth]{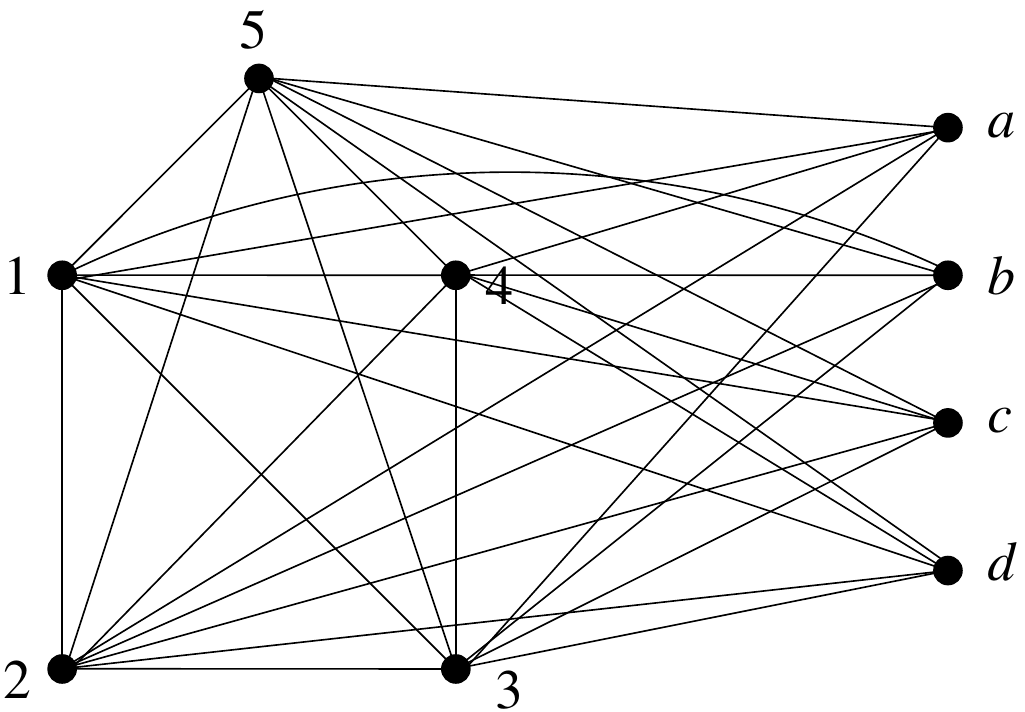}
}
\subfigure[The blocks of a $2$-$(\vv,\kk,1)$ generalized packing.]{
\includegraphics[width=\textwidth]{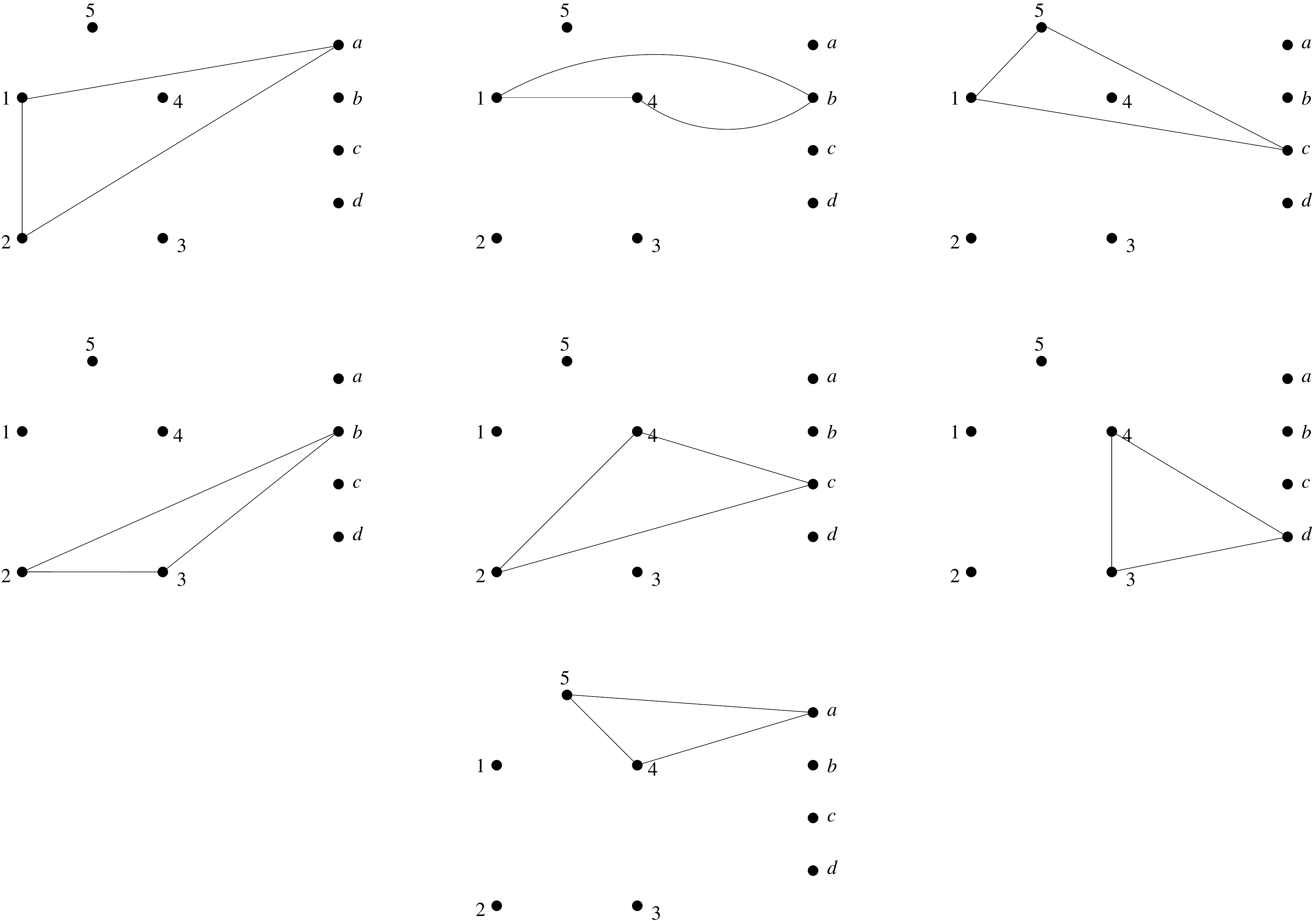}
}
\subfigure[The leave graph.]{
\includegraphics[width=0.3\textwidth]{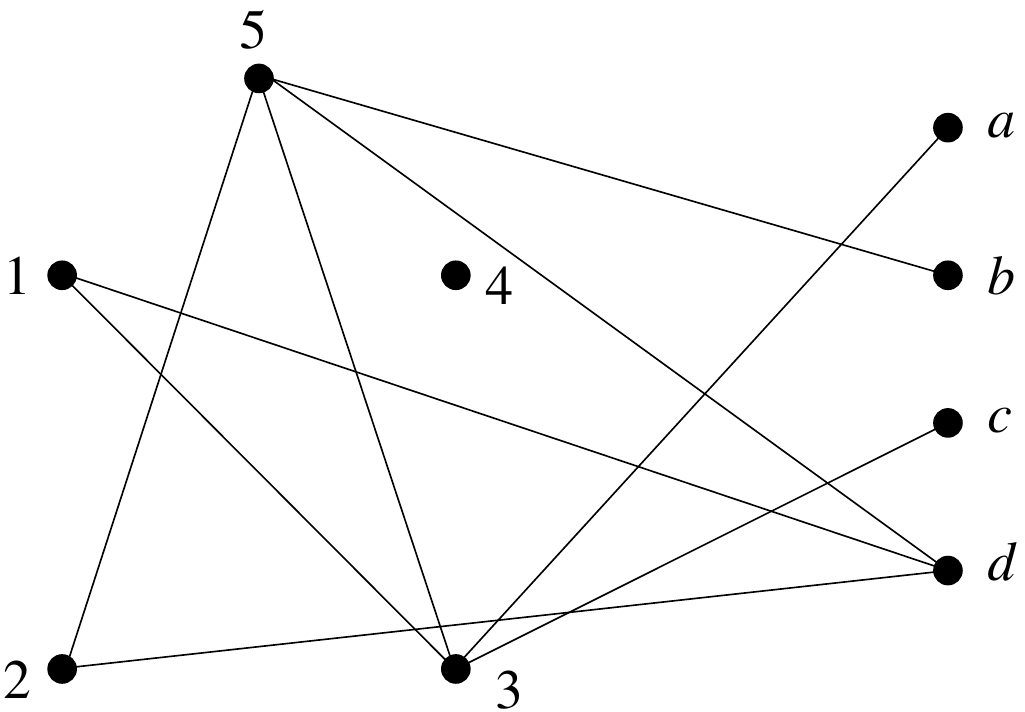}
}
\caption{The generalized packing of Example~\ref{example:K3packing}.}
\label{figure:leave}
\end{figure}
\end{example}

\subsection{Packing arrays and partial Latin squares: the case $\kk=(1,1\ldots,1)$} \label{section:PA}

One of Cameron's motivating examples in~\cite{cameron} for generalized $t$-$(\vv,\kk,\lambda)$ designs was the case $\kk=(1,1,\ldots,1)$, which (when $\vv=(s,s,\ldots,s)$) corresponds to orthogonal arrays.  Likewise, in~\cite{gencov} one of the motivating examples for generalized covering designs was covering arrays.  There is also a ``packing'' version of these objects, which we define now.

\begin{definition} \label{defn:packing_array}
Let $N,k,s,t,\lambda$ be positive integers.  A {\em packing array} $\PA_\lambda(N;k,s,t)$ is an $N\times k$ array with entries from an alphabet of size $s$, with the property that in every set of $t$ columns, any $t$-tuple of symbols from the alphabet occurs in at most $\lambda$ rows.
\end{definition}

Usually, we are interested in the case $\lambda=1$, and omit the subscript $\lambda$.  Note that such an array where every $t$-tuple occurs in {\em exactly} $\lambda$ rows is an {\em orthogonal array} (see the book by Hedayat et al.~\cite{hedayat}).  The typical question for packing arrays is to determine, for given values of $k$, $s$ and $t$, the largest $N$ such that there exists a $\PA(N;k,s,t)$: this value of $N$ is called the {\em packing array number}, and is denoted by ${\rm PAN}(k,s,t)$.  A listing of known packing array numbers is given in~\cite[Table III.3.123]{handbook}.

Unlike orthogonal arrays and covering arrays, not much attention has been paid to packing arrays in the literature, with the main references being the papers of Stevens and Mendelsohn~\cite{stevens_mendelsohn1,stevens_mendelsohn2}.  However, they arise as generalized packings in the same manner as did orthogonal arrays and covering arrays.

\begin{proposition} \label{prop:packing_array}
Let $\vv=(s,s,\ldots,s)$ and $\kk=(1,1,\ldots,1)$ have length~$k$.  Then a \mbox{$t$-$(\vv,\kk,\lambda)$} generalized packing (with $N$ blocks) is equivalent to a packing array $\PA_\lambda(N;k,s,t)$.
\end{proposition}

In particular, in the case where $\lambda=1$ and $t=2$, we have the following result.

\begin{proposition} \label{prop:packing_array2} 
The existence of the following objects are equivalent:
\begin{itemize}
\item[(i)] a packing array $\PA(N;k,s,2)$;
\item[(ii)] $k-2$ mutually orthogonal partial Latin squares of order $v$, each with the same $N$ cells filled;
\item[(iii)] a $2$-$(\vv,\kk,1)$ generalized packing with $N$ blocks, where $\vv = (s, s, \ldots, s)$ and $\kk = (1, 1, \ldots, 1)$ have length $k$.
\end{itemize}
\end{proposition}

Of course, we wish to consider arbitrary vectors $\vv$.  Without loss of generality, if $\kk=(1,1,\ldots,1)$, we may assume that $v_1\leq v_2 \leq \cdots \leq v_k$.  Now, in this case, the bound given by Proposition~\ref{t=2_bound} simplifies greatly as follows:
$$
D(\vv,\kk,2) \leq \min_{i\neq j} v_i v_j = v_1 v_2.
$$
To construct a generalized packing meeting this bound, we can use the same kind of idea as Proposition~\ref{prop:packing_array2}.  In particular, we use a particular class of partial Latin square, which we now define.

\begin{definition} \label{def:molr}
A $v_1 \times v_2$ {\em Latin rectangle}, where $v_1 \leq v_2$, is a $v_1 \times v_2$ array with $v_2$ symbols such that each symbol occurs exactly once in each row, and at most once in each column.  The $v_1 \times v_2$ Latin rectangles $L = (\ell_{ij})$ and $M=(m_{ij})$ are said to be {\em orthogonal} if $(\ell_{ij},m_{ij}) = (\ell_{i'j'},m_{i'j'})$ implies $i=i'$ and $j=j'$.  
We use the notation $\mathrm{MOLR}(v_1, v_2)$ to denote mutually orthogonal $v_1 \times v_2$ Latin rectangles.
\end{definition}

\begin{proposition} \label{max_PA_construction}
Let $\vv = (v_1, v_2, \ldots, v_k)$, where $v_1 \leq v_2 \leq \cdots \leq v_k$, and $\kk = (1, 1, \ldots, 1)$.  If there exist $k-2$ $\mathrm{MOLR}(v_1, v_2)$, then $D(\vv,\kk,2) = v_1 v_2$.
\end{proposition}

\begin{proof}
By Proposition~\ref{t=2_bound}, $D(\vv,\kk,2) \leq v_1 v_2$.  Suppose $L_1,L_2,\ldots,L_{k-2}$ are a collection of $k-2$ $\mathrm{MOLR}(v_1,v_2)$.  These rectangles give rise to a $2$-$(\vv',\kk,1)$ packing of size $v_1 v_2$, where $\vv'=(v_1, v_2, \ldots, v_2)$, by taking blocks of the form $( \{i\}, \{j\}, \{L_1(i,j)\}, \{L_2(i,j)\}, \ldots, \{L_{k-2}(i,j)\})$ where $1 \leq i\leq v_1$ and $1 \leq j \leq v_2$.  Hence, by Lemma~\ref{increase_v}, there is a $2$-$(\vv,\kk,1)$ packing of size $v_1 v_2$.
\end{proof}

Clearly, when there exist $k-2$ $\mathrm{MOLS}(v_2)$, we can use these to obtain the required $\mathrm{MOLR}(v_1,v_2)$, although there are examples of MOLR that do not arise from MOLS.  For example, there exist two $\mathrm{MOLR}(4,6)$~\cite{harvey_winterer}, from which we can construct a $2$-$(\vv,\kk,1)$ generalized packing, for $\vv=(4,6,v_3,v_4)$ (where $6\leq v_3 \leq v_4$) and $\kk=(1,1,1,1)$, with $4\times 6=24$ blocks.

The case where $k=4$, i.e.\ when we require two orthogonal Latin rectangles, is considered in detail in Section~\ref{section:1111}.

\subsection{Resolvability and block colouring: the case $t=2$, $\kk = (k-1,1)$}
Recall Proposition~\ref{prop:edgecolouring}, which showed that if $\kk=(2,1)$, a generalized packing is equivalent to an edge-coloured graph.  This idea holds more generally.

A {\em block colouring} of a block design is an assignment of colours to the blocks, so that blocks which intersect receive different colours.  As for graphs, the {\em chromatic index} of a design is the smallest number of colours needed for a block colouring.  We notice that, for a given colour, the blocks assigned that colour must all be disjoint; if these blocks contain all the points of the design, we call them a {\em parallel class}.  More generally, any collection of disjoint blocks is referred to as a {\em partial parallel class}; if all points except one appear, it is an {\em almost parallel class}.  A design where the blocks can be partitioned into parallel classes is said to be {\em resolvable}; the partition into parallel classes is called a {\em resolution} of the design.  (More information on resolvable designs can be found in~\cite[{\S}II.7]{handbook}.)

In~\cite{cameron}, Cameron observes that when $\kk = (k-1,1)$ and $\vv=(v_1,v_2)$, a generalized $2$-$(\vv,\kk,\lambda)$ design is equivalent to a resolvable $2$-$(v_1,k-1,\lambda)$ design; here, $v_2$ must equal the number of parallel classes.  Basically, a block in the generalized design consists of a block of the $2$-$(v-1,k-1,\lambda)$-design, with an element of $X_2$ indexing the parallel class it is in.  The same idea works for generalized packings.

\begin{proposition} \label{prop:resolvable}
Suppose $\vv=(v_1,v_2)$ and $\kk=(k-1,1)$.  Then a $2$-$(\vv,\kk,\lambda)$ generalized packing is equivalent to a $2$-$(v_1,k-1,\lambda)$ packing whose blocks are partitioned into at most $v_2$ partial parallel classes.  
\end{proposition}

Equivalently, such a generalized packing may be thought of as a $v_2$-block colouring of an ordinary $2$-$(v_1, k-1,\lambda)$ packing. 

In this case, the bounds from Proposition~\ref{indiv_part} and Corollary~\ref{schonheim_type} simplify as follows.

\begin{lemma} \label{(k-1,1) bound}
If $\kk = (k-1,1)$, then $D(\vv, \kk, 2) \leq \min \left\{ D(v_1, k-1, 2), 
{v_2}\left\lfloor v_1/(k-1)\right\rfloor \right\}$.
\end{lemma}

Later in the paper (in Sections~\ref{section:2_1} and~\ref{section:3_1}), we will see that this bound is always met when $k=3$, and usually met when $k=4$.

\subsection{Orthogonal colourings and orthogonal resolutions: the case $t=2$, $\kk=(k-2,1,1)$} \label{section:orthog_colouring}
In the case that $\kk=(k-2,1,1)$, the necessary conditions given in~\cite[Proposition 1]{cameron} assert that a $2$-$(\vv,\kk,1)$ design exists only if $\vv=\kk=(k-2,1,1)$, which is a trivial case.  Nevertheless, generalized packings with $\kk=(k-2,1,1)$ have interesting design-theoretical interpretations, and many objects in the literature arise as examples.

Let $\mathscr{D} = (X,\mathscr{B})$ be a $2$-$(v,k-2,1)$ packing, and let $f: \mathscr{B} \rightarrow \{1, 2, \ldots, s\}$ and $g: \mathscr{B} \rightarrow \{1, 2, \ldots, t\}$ be two proper block colourings of $\mathscr{D}$.  Let $F_1$, $F_2$, $\ldots$, $F_s$ be the colour classes of the block colouring $f$ and let $G_1$, $G_2$, $\ldots$, $G_t$ be the colour classes of $g$.  We say that block colourings $f$ and $g$ are {\em orthogonal} if $|F_i \cap G_j| \leq 1$ for any $i \in \{1, \ldots, s\}$ and $j \in \{ 1, \ldots, t\}$.  That is, if two blocks receive the same colour in one of the colourings, then they must receive different colours in the other.

From colourings $f$ and $g$, we may create an $s \times t$ array $A$ in the following manner: for each block $B$, we place $B$ in the $(i,j)$-entry of $A$ if $f(B) = i$ and $g(B) = j$.  The array $A$ has the following properties:
\begin{itemize}
\item[(i)] each entry of $A$ is either empty or else contains a $(k-2)$-subset of $X$;
\item[(ii)] each symbol in $X$ appears at most once in each row and at most once in each column;
\item[(iii)] each pair of elements occurs at most once as a subset of an entry of $A$.
\end{itemize}
Conversely, it is easy to see that given an $s \times t$ array $A$ satisfying properties 1, 2 and 3, then by letting $\mathscr{B}$ be the set of (nonempty) entries of $A$, and for each $B \in \mathscr{B}$, $f(B)=i$ and $g(B)=j$, where $B$ appears in the $(i,j)$-entry of $A$, then we obtain two orthogonal block colourings of the packing $(X,\mathscr{B})$.

Moreover, an $s \times t$ array $A$ satisfying properties 1, 2 and 3 is equivalent to a $2$-$(\vv,\kk,1)$ packing, where $\vv=(v,s,t)$ and $\kk=(k-2,1,1)$, with blocks of the form $(B,i,j)$, where $B$ is the $(i,j)$-entry of $A$.  Thus, a $2$-$(\vv,\kk,1)$ packing is equivalent to the existence of two orthogonal block colourings of a $2$-$(v,k-2,1)$ packing, with $s$ and $t$ colour classes.  

In the case where each block colouring is a resolution of the design, we refer to them as {\em orthogonal resolutions}, and the design is said to be {\em doubly resolvable}.  For example, a doubly resolvable Steiner triple system is known as a {\em Kirkman square} (see Colbourn et al.~\cite{colbourn_et_al}); this is a $2$-$(\vv,\kk,1)$ generalized packing with $\vv=(v,r,r)$ (where $r=(v-1)/2$) and $\kk=(3,1,1)$.  The name arises as the blocks are arranged in an $r\times r$ square array.  The smallest known example of a Kirkman square is for $v=27$: see~\cite[Figure 1]{colbourn_et_al}.

The case where $\kk=(2,1,1)$ is considered in detail in Section~\ref{(2,1,1) case}.

\section{The case $t=2$ and $k=3$} \label{section:blocksize3}

In~\cite{cameron}, Cameron's motivating examples were generalized 2-designs where $k=3$: these correspond to Steiner triple systems, 1-factorizations of complete graphs and Latin squares.  We extend this characterization to generalized packings where $k=3$, where these provide prototypical examples for the three possibilities for $\kk$.

\subsection{$\kk = (3)$: ordinary packings}
When $\kk=(3)$, a generalized $2$-$(\vv,\kk,1)$ packing is equivalent to an ordinary $2$-$(v,3,1)$ packing, also sometimes known as a {\em partial Steiner triple system}.  In this case, the packing numbers, structure of maximum packings, and leave graphs are all known, and were determined in the 1966 paper of Sch\"onheim~\cite{schonheim}.  They depend on congruences modulo~6, and are summarised in Table~\ref{table:triples} below (taken from~\cite[Table VI.40.22]{handbook}).
\begin{table}[hbt]
\centering
\renewcommand{\arraystretch}{1.2}
\begin{tabular}{|c|c|c|} \hline
$v \equiv$   & $D(v,3,2)$     & Structure of leave graph \\ \hline
$1,3\pmod 6$ & $v(v-1)/6$     & empty \\ \hline
$0,2\pmod 6$ & $v(v-2)/6$     & 1-factor \\ \hline
$4 \pmod 6$  & $(v^2-2v-2)/6$ & $K_{1,3}$ and a matching of size $(v-4)/2$\\ \hline
$5 \pmod 6$  & $(v^2-v-8)/6 $ & 4-cycle \\ \hline
\end{tabular}
\renewcommand{\arraystretch}{1.0}
\caption{Maximum $2$-$(v,3,1)$ packings.}
\label{table:triples}
\end{table}

Of course, the cases $v\equiv 1,3 \pmod 6$ are Steiner triple systems.  A detailed description of the constructions of maximum packings can be found in Chapter 4 of Lindner and Rodger~\cite{lindner_rodger}.  We note that $D(v,3,2)$ meets the Johnson--Sch\"onheim bound (Proposition~\ref{johnson_schonheim}) with equality when $v \equiv 0,1,3 \pmod 6$, and is 1 less than the Johnson--Sch\"onheim bound otherwise.  See also the survey by Mills and Mullin~\cite{mills_mullin}, where the case $\lambda>1$ is also described: this was solved by Hanani~\cite{hanani}.

\subsection{$\kk = (2,1)$: edge-colourings and factorizations} \label{section:2_1}
Recall from Proposition~\ref{prop:edgecolouring} that if $\kk=(2,1)$, a $2$-$(\vv,\kk,1)$ generalized packing corresponds to a proper edge-colouring of a graph.  To construct maximum generalized packings, it helps to consider graphs, and edge-colourings, with some structure.

We will show that the bound given in Lemma~\ref{(k-1,1) bound} can be achieved.  When $\kk=(2,1)$, that reduces to 
$$
D(\vv, \kk, 2) \leq \min \left\{ \binom{v_1}{2}, {v_2}\left\lfloor \frac{v_1}{2}\right\rfloor \right\}.
$$

When a generalized 2-design exists for $\kk = (2,1)$, it is equivalent to a 1-factorization of a complete graph $K_{v_1}$.  (One can regard each 1-factor as a colour, which appears at every vertex.)  This occurs when $v_1$ is even and $v_2 = v_1-1$.  We can extend this idea to obtain maximum generalized packings for arbitrary $\vv=(v_1,v_2)$.

\begin{proposition} \label{prop:2_1_case}
Suppose $\vv=(v_1,v_2)$ and $\kk=(2,1)$.  Then there exists a $2$-$(\vv,\kk,1)$ generalized packing meeting the bound given in Lemma~\ref{(k-1,1) bound}.
\end{proposition}

\begin{proof}
First, we suppose $v_1$ is even.  If $v_2 = v_1-1$, a 1-factorization of $K_{v_1}$ gives us a generalized 2-design, and thus a a generalized packing whose leave is empty.  If $v_1$ is even and $v_2 \geq v_1-1$, then we use the same design, with the excess vertices in $X_2$ not used in any block.  The size of the design is $\binom{v_1}{2} = D(v_1,2,2)$, and the leave graph is $K_{v_1, v_2-v_1+1}$.  
If $v_1$ is even and $v_2 < v_1-1$, then we use $v_2$ of the 1-factors in a 1-factorization of $K_{v_1}$, with the $v_2$ 1-factors indexed by $X_2$.  The size of the packing is $v_1 v_2/2$, and the leave graph consists of the union of $(v_1-v_2-1)$ 1-factors in $K_{v_1}$; its precise structure is dependent on the choice of the 1-factorization.   

Next, suppose that $v_1$ is odd.  In this case, there exists a near 1-factorization $\mathcal{F}$ of $K_{v_1}$, containing $v_1$ matchings of size $(v_1-1)/2$.  If $v_2 = v_1$, then we take as a block an edge of $K_{v_1}$, together with an element of $X_2$ indexing the near 1-factor of $\mathcal{F}$ to which the edge belongs.  In this way, we achieve a packing of size $D(v_1,2,2)$ whose leave consists of a matching of size $v_1$ between $X_1$ and $X_2$.
If $v_2 > v_1$, then we use the same design, with the excess vertices in $X_2$ not occurring in any block.  Again, the design has size $D(v_1,2,2)$.  
If $v_2 < v_1$, then we take $v_2$ of the near 1-factors in $\mathcal{F}$.  In this case, the design has size $v_2 (v_1-1)/2 = v_2 \lfloor v_1/2\rfloor$.  
\end{proof}

\subsection{$\kk = (1,1,1)$: Latin rectangles}
This is the easiest case.  Let $\vv=(v_1, v_2, v_3)$, where $v_1 \leq v_2 \leq v_3$.  By Corollary~\ref{schonheim_type}, we have that $D(\vv,\kk,2) \leq v_1 v_2$.  The existence of a $v_1 \times v_2$ Latin rectangle guarantees the existence of a generalized packing of size $v_1 v_2$ (cf.\ Proposition~\ref{max_PA_construction}).  The graphical interpretation of Section~\ref{section:graphical} in this case is also straightforward: it is simply a packing of 3-cycles into a complete multipartite graph $K_{v_1,v_2,v_3}$.

\section{The case $t=2$ and $k=4$} \label{section:blocksize4}

The bulk of this paper is devoted to constructing optimal generalized packings where $t=2$ and $k=4$.  There are five cases, corresponding to the five partitions of~4.  Cameron~\cite{cameron} showed that generalized 2-designs can only exist in three of these cases, namely $\kk=(4)$, $(3,1)$ and $(1,1,1,1)$: these correspond to $2$-$(v,4,1)$ designs, Kirkman triple systems and pairs of orthogonal Latin squares, respectively.  As we did for $k=3$, we extend this characterization to generalized packings with $k=4$.  This also requires us to study the cases $\kk=(2,1,1)$ and $\kk=(2,2)$; while no generalized 2-design exists in those cases, there are still (at least when $\kk=(2,1,1)$) plenty of examples of generalized packings which exist in the literature.

\subsection{$\kk = (4)$}
A $2$-$((v),(4),\lambda)$ generalized packing is equivalent to an ordinary $2$-$(v,4,\lambda)$ packing.  For $\lambda=1$, the packing numbers were determined in 1978 by Brouwer \cite{brouwer_packings}.  Let $U(v,4,2)$ denote the Johnson--Sch\"onheim bound (see (Proposition~\ref{johnson_schonheim})).

\begin{theorem}[Brouwer~\cite{brouwer_packings}] \label{thm:quadruples}
The packing number $D(v,4,2) = U(v,4,2)-\epsilon$, where
$$
\epsilon = \left\{
\begin{array}{ll}
1, & \textrm{if } v\equiv 7,\, 10 \pmod{12},\,\, v\neq 10,19 \\
1, & \textrm{if } v\equiv 9,\, 17 \pmod{12} \\
2, & \textrm{if } v=8,\, 10,\, 11 \\
3, & \textrm{if } v=19 \\
0, & \textrm{otherwise.}
\end{array} \right.
$$
\end{theorem}
Brouwer also gave constructions in all cases: these are listed in~\cite[Table VI.40.23]{handbook}.  We note that $2$-$(v,4,1)$ designs exist exactly when $v\equiv 1,4 \pmod{12}$.

If $\lambda>1$, the packing numbers have also been determined completely: this is due to work of Billington et al.~\cite{billington}, Hartman~\cite{hartman} and Assaf~\cite{assaf}.  (The reader is referred to Mills and Mullin~\cite{mills_mullin} for full details.)

\subsection{$\kk = (3,1)$} \label{section:3_1}

Recall that in the case $\kk=(k-1,1)$, then a $2$-$(\vv,\kk,\lambda)$ generalized packing corresponds to a proper colouring of a $2$-$(v_1,k-1,1)$ packing using $v_2$ colours, where a block $(B,\{i\})$ in the generalized packing tells us that in the $2$-$(v_1,k-1,1)$ packing, block $B$ is assigned colour~$i$.  (We will refer to a generalized packing in this case by listing the colour classes of blocks in the corresponding ordinary packing.)  In this section, we will show that if $\kk=(3,1)$, $t=2$ and $\lambda=1$, then the bound given in Lemma~\ref{(k-1,1) bound} can be achieved in many cases.  This simplifies as follows.

\begin{proposition} \label{prop:31bound}
Suppose $\vv=(v_1,v_2)$ and $\kk=(3,1)$ with $\vv\geq\kk$.  Then
$$
D(\vv,\kk,2) \leq \min\{ D(v_1,3,2), v_2 \lfloor v_1/3\rfloor\}.
$$
\end{proposition}

If $\kk=(3,1)$, the only possibility for a generalized 2-design is if $v_1 \equiv 3 \pmod{6}$ and $v_2 = (v_1-1)/2$ (see~\cite[subsection 3.2.2]{cameron}).  Such a design corresponds to a {\em Kirkman triple system} on $v_1$ points, i.e.\ a Steiner triple system with a resolution into parallel classes, and denoted $\kts(v_1)$.  These originate in a problem of Kirkman from 1850~\cite{kirkman}, often known as {\em Kirkman's schoolgirls problem}, for the case $v_1=15$.  
Kirkman solved the problem for $v_1=15$ himself the following year~\cite{kirkman1851}, but the existence of Kirkman triple systems for all values of $v_1 \equiv 3 \pmod{6}$ was not settled until the late 1960s, when it was shown by Ray-Chaudhuri and Wilson~\cite{RC_wilson}.  (A survey on Kirkman triple systems and related designs can be found in Stinson~\cite{stinson_kirkman}.)

More generally, if we have a block colouring of $2$-$(v_1,3,1)$ packing with $s$ colours, its {\em colour type} is the sequence $(m_1,\ldots,m_s)$, where $m_i$ is the number of blocks assigned colour~$i$.  We may also use exponential notation: colour type $w_1^{\alpha_1} w_2^{\alpha_2} \cdots w_n^{\alpha_n}$ means that there are $\alpha_i$ colour classes of size $w_i$ for each $i=1,2,\ldots,n$, and so $\sum \alpha_i =s$.

Motivated by an application to unipolar communications, Colbourn and Zhao~\cite{colbourn_zhao} introduced the notion of a {\em Kirkman signal set}, which is a $2$-$(v_1,3,1)$ packing partitioned into $s$ colour classes of size $m$ (that is, it has colour type $m^s$), where $s$ is as large as possible.  If $m=\lfloor v_1/3 \rfloor$ (and so is also as large as possible), we have a {\em maximum Kirkman signal set}.  More recently, Colbourn, Horsley and Wang~\cite{colbourn_horsley_wang} introduced a {\em strong Kirkman signal set} (denoted $\skss(v_1)$) to be a block colouring of a maximum $2$-$(v_1,3,1)$ packing of colour type $m^s r^1$, where $D(v_1,3,1)=sm+r$ and $r<m$.  That is, an $\skss(v_1)$ has $s$ colour classes of size $m=\lfloor v_1/3 \rfloor$, and the remaining $r$ blocks all receive the same colour.  For example, a $\kts(v_1)$ is an $\skss(v_1)$ when $v_1 \equiv 3 \pmod 6$; in this case, $r=0$.

For $v_1 \equiv 0,1,2 \pmod 6$, various other known objects arise as $\skss(v_1)$.  When $v_1 \equiv 2 \pmod 6$, an $\skss(v_1)$ can be obtained by deleting a point from a $\kts(v_1+1)$; here $r=0$ also.  When $v_1 \equiv 1 \pmod 6$, an $\skss(v_1)$ is a {\em Hanani triple system}, namely an $\sts(v_1)$ whose triples may be partitioned into $s=(v_1-1)/2$ almost parallel classes and one partial parallel class of size $r=(v_1-1)/6$.  These were introduced in a 1993 paper of Vanstone et al.~\cite{vanstone_et_al}, who showed that such a system exists if and only if $v \equiv 1 \pmod{6}$ and $v_1 \notin \{7, 13\}$.  When $v_1 \equiv 0 \pmod 6$, an $\skss(v_1)$ is a {\em nearly Kirkman triple system}, which is a colouring of a maximum packing of triples on $v_1$ points, with $s=(v_1-2)/2$ colour classes of size $m=v_1/3$; once again we have $r=0$.  These were introduced in 1974 by Kotzig and Rosa~\cite{kotzig_rosa}.  In 1977, it was shown by Baker and Wilson~\cite{baker_wilson} that there exists such a system if and only if $v_1 \equiv 0 \pmod{6}$ and $v_1 \geq 18$, with three possible exceptions.  Two of the exceptional cases were later solved by Brouwer~\cite{brouwer_kirkman}, and the remaining case by Rees and Stinson~\cite{rees_stinson}.

The remaining possibilities, namely $v_1 \equiv 4,5 \pmod 6$, were dealt with by Colbourn, Horsley and Wang~\cite{colbourn_horsley_wang} (see also~\cite{CHW_trails}).\footnote{Some of the results of Colbourn, Horsley and Wang~\cite{CHW_trails,colbourn_horsley_wang} for $v_1 \equiv 5 \pmod 6$ were independently obtained by the authors of the present paper, using exactly the same approach; however, their results were submitted for publication before the present authors became aware of their work.}  Combined with the results of the previous paragraphs, they proved the following.

\begin{theorem}[Colbourn, Horsley and Wang%
~{\cite[Theorem 2.4]{colbourn_horsley_wang}}%
] \label{thm:CHW}
Suppose $v_1\geq 3$.  Then there exists a strong Kirkman signal set $\skss(v_1)$ unless $v_1 \in \{6,7,11,12,13\}$.\footnote{The value $v_1=11$ is missing from the statement of~\cite[Theorem 2.4]{colbourn_horsley_wang}; however, it is addressed elsewhere in their paper.}
\end{theorem}

The connection between strong Kirkman signal sets and generalized packings is established in the result below.

\begin{theorem} \label{thm:SKSS}
Suppose $\vv=(v_1,v_2)$ and $\kk=(3,1)$.  If there exists a strong Kirkman signal set $\skss(v_1)$, then there exists a $2$-$(\vv,\kk,1)$ generalized packing meeting the bound of Proposition~\ref{prop:31bound}.
\end{theorem}

\begin{proof}
First, we recall that the bound of Proposition~\ref{prop:31bound} is
$$
D(\vv, \kk, 2) \leq \min\left\{ D(v_1, 3, 2), v_2 \left\lfloor \frac{v_1}{3} \right\rfloor \right\}.
$$
Now suppose there exists an $\skss(v_1)$, which has colour type $m^s r^1$, where $m=\lfloor v_1/3 \rfloor$ and $D(v_1,3,2)=ms+r$.  If $v_2 \leq s$, we take $v_2$ of the colour classes of size $m$, which yield a generalized packing of size $v_2 \lfloor v_1/3 \rfloor$.  If $v_2 > s$, then the $\skss(v_1)$ itself is a $2$-$(\vv,\kk,1)$ generalized packing of size $D(v_1,3,2)$.
\end{proof}

Thus Theorem~\ref{thm:CHW} establishes the existence of generalized packings for $\kk=(3,1)$ unless $v_1 \in \{6,7,11,12,13\}$.  
In~\cite[{\S}3.1]{colbourn_horsley_wang}, Colbourn, Horsley and Wang determine all the possible colour types for $2$-$(v_1,3,1)$ packings where $v_1\leq 13$.  From this, it is possible to determine the value of $D(\vv,\kk,2)$ for $\kk=(3,1)$ in the five exceptional cases.\footnote{Solutions in all but three of these exceptional cases, namely $\vv=(11,6)$, $(12,6)$ and $(13,6)$, were also determined independently by the present authors.}  We now consider each of these exceptions, beginning with $v_1=6$.

\begin{lemma}
Let $\vv=(6,v_2)$ and $\kk=(3,1)$.  Then
$$
D(\vv, \kk, 2) = 
\left\{
\begin{array}{ll}
2, & \mbox{if } v_2=1, 2 \\
3, & \mbox{if } v_2=3 \\
4, & \mbox{if } v_2 \geq 4.
\end{array}
\right.
$$
\end{lemma}

This is straightforward to show, so we leave the proof as an exercise.  Almost as straightforward is the case $v_1=7$, which we do next.

\begin{lemma}
Let $\vv=(7,v_2)$ and $\kk=(3,1)$.  Then 
$$
D(\vv,\kk,2) = \left\{ 
\begin{array}{ll}
2, & \mbox{if } v_2=1 \\
3, & \mbox{if } v_2=2 \\
4, & \mbox{if } v_2=3 \\
5, & \mbox{if } v_2=4, 5 \\
6, & \mbox{if } v_2=6 \\
7, & \mbox{if } v_2 \geq 7.
\end{array}
\right.
$$ 
\end{lemma}

\begin{proof} The unique maximum $2$-$(7,3,1)$ packing is, of course, the Fano plane, which has chromatic index~7: since any two blocks intersect, each must receive its own colour.  So if $v_2\geq 7$, we are done.  Also, if $v_2=6$, we obtain a maximum generalized packing by taking six of the blocks from the Fano plane.

Now, with seven points, it can be shown that there can be at most one colour class of size 2, and the maximum size of a packing containing such a colour class is~5.  The remaining results follow from this observation.
\end{proof}

The remaining exceptions are slightly more involved.  We continue with the case $v_1=11$.

\begin{lemma}
Let $\vv=(11,v_2)$ and $\kk=(3,1)$.  Then
$$
D(\vv,\kk,2) = 
\left\{
\begin{array}{ll}
3v_2, & \mbox{if } v_2 \leq 5 \\
16 ,  & \mbox{if } v_2=6 \\
17,   & \mbox{if } v_2 \geq 7.
\end{array}
\right.
$$
Thus, the packing number $D(\vv,\kk,2)$ meets the bound of Proposition~\ref{prop:31bound}, except when $v_2 = 6$.
\end{lemma}

\begin{proof}
Note that Proposition~\ref{prop:31bound} gives an upper bound of $D(\vv,\kk,2) \leq \min\{17,3v_2\}$.  It was shown by Colbourn and Rosa~\cite{colbourn_rosa85} that a maximum $2$-$(11,3,1)$ packing has chromatic index~7, so if $v_2\geq 7$, we are done.  This also implies that for $v_2=6$, the maximum size of a generalized packing is at most~16.  Furthermore, Colbourn, Horsley and Wang obtained a colouring of the 17 blocks with colour type $3^5 1^2$ (an example is given in Appendix~\ref{section:31appendix}, Example~\ref{example:11pts}); the remaining results follow from the existence of this.
\end{proof}

\begin{lemma}
Let $\vv=(12,v_2)$ and $\kk=(3,1)$.  Then
$$
D(\vv,\kk,2) = 
\left\{
\begin{array}{ll}
4v_2, & \mbox{if } v_2 \leq 4 \\
19 ,  & \mbox{if } v_2=5,6 \\
20,   & \mbox{if } v_2 \geq 7.
\end{array}
\right.
$$
Thus, the packing number $D(\vv,\kk,2)$ meets the bound of Proposition~\ref{prop:31bound}, except when $v_2 = 5$ or~6.
\end{lemma}

\begin{proof}
Proposition~\ref{prop:31bound} gives the upper bound $D(\vv,\kk,2) \leq \min \{20,4v_2\}$.  Now, as there is no $\nkts(12)$, there cannot exist a 5-block colouring of a maximum $2$-$(12,3,1)$ packing.  Furthermore, the enumeration of colour types by Colbourn, Horsley and Wang shows that the chromatic index is in fact~7.  Thus the maximum number of blocks in a generalized packing for $v_2=5,6$ is at most~19.  

Now, there exists a colouring of 20 blocks with colour type $3^6 2^1$ and a colouring of 19 blocks with colour type $4^4 3^1$, given in Appendix~\ref{section:31appendix}, Examples~\ref{example:20blocks} and~\ref{example:19blocks} respectively; these were obtained independently by the present authors.  For $v_2\geq 7$, the result follows from the existence of the former, and for $v_2\leq 6$, it follows from the latter.
\end{proof}

\begin{lemma}
Let $\vv=(13,v_2)$ and $\kk=(3,1)$.  Then
$$
D(\vv,\kk,2) = 
\left\{
\begin{array}{ll}
4v_2, & \mbox{if } v_2 \leq 6 \\
25,   & \mbox{if } v_2=7 \\
26,   & \mbox{if } v_2 \geq 8.
\end{array}
\right.
$$
Thus, the packing number $D(\vv,\kk,2)$ meets the bound of Proposition~\ref{prop:31bound}, except when $v_2 = 7$.
\end{lemma}

\begin{proof}
This time, the upper bound of Proposition~\ref{prop:31bound} works out as $D(\vv,\kk,2) \leq \min\{ 26, 4v_2\}$.  Now, a maximum $2$-$(13,3,1)$ packing is a Steiner triple system: there are exactly two $\sts(13)$, and each has chromatic index~8 (see~\cite{mathon_phelps_rosa}).  Thus for $v_2 \geq 8$, we are done, while for $v_2=7$, the best we can hope for is 25~blocks.  Fortunately, there is an $\sts(13)$ with colour type $4^4 3^3 1^1$ (see Appendix~\ref{section:31appendix}, Example~\ref{example:sts13}), and deleting the colour class of size~1 yields a maximum generalized packing if $v_2=7$.

For $v_2\leq 6$, there is a packing with colour type~$4^6$ (see Appendix~\ref{section:31appendix}, Example~\ref{example:13pts}) obtained by Colbourn, Horsley and Wang: taking $v_2$ colour classes from this gives maximum generalized packings.
\end{proof}

We pull together all the results above for $\kk=(3,1)$ in the following theorem.

\begin{theorem} \label{thm:3_1_summary}
Suppose $\vv=(v_1,v_2)$ and $\kk=(3,1)$, where $v_1\geq 3$.  Then there exists a $2$-$(\vv,\kk,1)$ generalized packing meeting the bound of Proposition~\ref{prop:31bound}, except for the specific values listed in Table~\ref{table:3_1_exceptions} below.
\end{theorem}

\begin{table}[hbt]
\centering
\renewcommand{\arraystretch}{1.2}
\begin{tabular}{|c|c|c|} \hline
$\vv=(v_1,v_2)$  & Projected bound & Packing number \\ \hline
$(6,2)$          & 4               & 2 \\
$(6,3)$          & 4               & 3 \\ \hline
$(7,2)$          & 4               & 3 \\
$(7,3)$          & 6               & 4 \\
$(7,4)$, $(7,5)$ & 7               & 5 \\
$(7,6)$          & 7               & 6 \\ \hline
$(11,6)$         & 17              & 16 \\ \hline
$(12,5)$, $(12,6)$ & 20            & 19 \\ \hline
$(13,7)$         & 26              & 25 \\ \hline
\end{tabular}
\renewcommand{\arraystretch}{1.0}
\caption{Exceptions for when $D(\vv,(3,1),2)$ does not meet the bound of Proposition~\ref{prop:31bound}.}
\label{table:3_1_exceptions}
\end{table}

\subsection{$\kk = (2,1,1)$} \label{(2,1,1) case}
This is the first case for $k=4$ where no generalized $2$-$(\vv,\kk,1)$ design exists to provide a starting point for us (apart from the trivial case where $\vv=(2,1,1)$).  However, as mentioned in Section~\ref{section:orthog_colouring}, interesting objects still arise.

When $\kk=(2,1,1)$, we let $\vv=(v_1, v_2, v_3)$ and note that we may assume without loss of generality that $v_2 \leq v_3$.  Proposition~\ref{t=2_bound} gives us the following upper bound on the packing number in this case.
\begin{proposition} \label{211-bound}
Let $\vv=(v_1, v_2, v_3)$, where $v_2 \leq v_3$ and $\kk=(2,1,1)$.  Then
$$
D(\vv,\kk,2) \leq \min \left\{ \binom{v_1}{2}, v_2 \left\lfloor \frac{v_1}{2}\right\rfloor, v_2 v_3 \right\}.
$$
\end{proposition}

In Section~\ref{section:orthog_colouring}, we saw that $2$-$(\vv,\kk,1)$ generalized packings with $\kk=(k-2,1,1)$ can be described in terms of orthogonal colourings.  In this case, we observe that a $2$-$(\vv,\kk,1)$ generalized packing with $\vv=(v_1,v_2,v_3)$ and $\kk=(2,1,1)$ is precisely equivalent to a pair of {\em orthogonal edge colourings} of a graph on $v_1$ vertices, where the two colourings use at most $v_2$ and at most $v_3$ colours respectively.  This concept was introduced by Archdeacon, Dinitz and Harary in 1985~\cite{archdeacon_dinitz_harary}, but does not appear to be particularly well-known.  (The case where each edge colouring is a 1-factorization is better-known: see Alspach et al.~\cite{alspach_et_al}, for instance.)  However, an alternative interpretation is as follows.

\begin{lemma} \label{array}
Let $S$ be a set of size $v_1$.  The existence of a $2$-$(\vv, \kk, 1)$ generalized packing with $N$ blocks, where $\vv = (v_1, v_2, v_3)$ and $\kk = (2,1,1)$ is equivalent to the existence of a $v_2 \times v_3$ array $A$ with the following properties:
\begin{itemize}
\item[(i)] each cell is either empty or else contains an unordered pair of elements from $S$;
\item[(ii)] exactly $N$ cells are non-empty;
\item[(iii)] each symbol appears at most once in each row and at most once in each column of $A$;
\item[(iv)] each pair of symbols appears in at most one cell.
\end{itemize}
\end{lemma}

\begin{proof}
Suppose we have a $2$-$(\vv, \kk, 1)$ packing of size $b$, and suppose that $X_1 =S$, $X_2 = \{x_1, x_2, \ldots, x_{v_2}\}$ and $X_3 = \{y_1, y_2, \ldots, y_{v_3}\}$.  We form the desired array $A$ by, for each block $\{ \{ s, s' \}, \{ x_i \}, \{y_j\} \}$, placing the pair $\{s, s'\}$ in cell $(i,j)$.  It is easy to verify that $A$ has the desired properties.  

Conversely, suppose we have an array $A$ of the type described.  We form a collection of blocks $\mathscr{B}$ by taking all blocks of the form $\{ \{ s, s'\}, \{x_i\}, \{y_j\}\}$, where the $(i,j)$-entry of $A$ is nonempty and contains the pair $\{s,s'\}$.  These blocks form the desired generalized packing. 
\end{proof}

\begin{example} \label{howell_example}
The array
$$
\begin{array}{|c|c|c|c|} \hline
1,2 & 3,4 & 5,6 &    \\ \hline
   & 1,6 & 2,3 & 4,5 \\ \hline
3,5 &    & 1,4 & 2,6 \\ \hline
4,6 & 2,5 &    & 1,3 \\ \hline
\end{array}
$$
is equivalent to the $2$-$((6,4,4),(2,1,1),1)$ packing with the following blocks:
$$
\begin{array}{llll}
(\{1,2\},\{1\},\{1\}) & (\{1,6\},\{2\},\{2\}) & (\{3,5\},\{3\},\{1\}) & (\{4,6\},\{4\},\{1\}) \\
(\{3,4\},\{1\},\{2\}) & (\{2,3\},\{2\},\{3\}) & (\{1,4\},\{3\},\{3\}) & (\{2,5\},\{4\},\{2\}) \\
(\{5,6\},\{1\},\{3\}) & (\{4,5\},\{2\},\{4\}) & (\{2,6\},\{3\},\{4\}) & (\{1,3\},\{4\},\{4\}). 
\end{array}
$$
\end{example}

In this section, we will describe generalized packings with $\kk=(2,1,1)$ in terms of the associated array, as given in Lemma~\ref{array}.  In particular, we will construct maximum generalized packings by showing the existence of such an array.  In most cases, our constructions will use arrays known as Howell designs, introduced in a 1974 paper of Hung and Mendelsohn~\cite{hung_mendelsohn}, which we now define.

\begin{definition}
Let $s$ and $n$ be integers.  A {\em Howell design} $H(s,2n)$ is an $s \times s$ array satisfying the following properties:
\begin{itemize}
\item[(i)] each cell is either empty or contains an unordered pair of symbols chosen from an alphabet of size $2n$;
\item[(ii)] each symbol appears exactly once in each row and column;
\item[(iii)] each pair of symbols appears in at most one cell.
\end{itemize}
\end{definition}

Note that since each of the $s$ rows contains $n$ filled cells, the total number of nonempty cells in a Howell design $H(s,2n)$ is $sn$.  The question of the existence of Howell designs was settled in two papers from the 1980s, and is stated below.

\begin{theorem} [Anderson, Schellenberg and Stinson \cite{anderson_schellenberg_stinson}; Stinson~\cite{Stinson82}] \label{howell}
There exists a Howell design $H(s,2n)$ if and only if $s+1 \leq 2n \leq 2s$ and $(s,2n) \neq (2,4)$, $(3,4)$, $(5,6)$ or $(5,8)$.
\end{theorem}

There are two extreme cases of Howell designs that are worth mentioning here.  The first is the case $s=n$, where an $H(n,2n)$ is known as a $\mathrm{SOMA}(2,n)$ (see~\cite{soicher_SOMA}).  (The name is an acronym for {\em simple orthogonal multi-array}, and is due to Phillips and Wallis~\cite{phillips_wallis}.)  In this case, every cell is filled, and a $\mathrm{SOMA}(2,n)$ may be obtained by superimposing two $\mathrm{MOLS}(n)$ with disjoint symbol sets; such a SOMA is said to be {\em Trojan}~\cite{bailey_SLS}.  Of note is the existence of a $\mathrm{SOMA}(2,6)$, first shown by Hung and Mendelsohn~\cite{hung_mendelsohn}; several examples are now known~\cite{bailey_SLS,bailey_howell,seah_stinson}.  However, it is not difficult to see that there is no $\mathrm{SOMA}(2,2)$.

The second extreme case is when $s=2n-1$, where an $H(2n-1,2n)$ is known as a {\em Room square} of side $2n-1$, denoted $\mathrm{RS}(2n-1)$, after T.~G.~Room~\cite{room}.  The existence of Room squares is less straightforward to demonstrate: this was done by Mullin and Wallis in 1975~\cite{mullin_wallis}, who showed that there exists an $\mathrm{RS}(2n-1)$ if and only if $2n-1\geq 7$.  Detailed information on Room squares can be found in the survey by Dinitz and Stinson~\cite{dinitz_stinson}.  Note that in a Room square, every possible pair of symbols appears in exactly one cell.

In light of Lemma~\ref{array}, the following application of Howell designs is clear.  (Note that if $\vv=(2n,s,s)$ and $\kk=(2,1,1)$, the bound of Proposition~\ref{211-bound} implies that $D(\vv,\kk,2) \leq sn$.)

\begin{lemma} \label{howell_lemma}
Let $\vv = (2n,s,s)$ and $\kk = (2,1,1)$.  If there exists a Howell design $H(s,2n)$, then there exists a $2$-$(\vv,\kk,1)$ generalized packing of size $sn$, and so meeting the bound of Proposition~\ref{211-bound}.
\end{lemma}

Example~\ref{howell_example} provides an illustration of Lemma~\ref{howell_lemma}, in which we construct a $2$-$(\vv,\kk,1)$ generalized packing of size 12 from a Howell design $H(4,6)$.

For other vectors $\vv$, Howell designs nevertheless provide the prototypical example of a generalized packing on which we base our constructions.  Recall that a generalized packing is equivalent to an array similar to a Howell design on $v_1$ symbols (which may be thought of as a partial Howell design, or a Howell packing).  If $v_1$ is even, our constructions frequently start with a Howell design on $v_1$ symbols.  However, if $v_1$ is odd, there is no Howell design on $v_1$ symbols.  Our strategy in such cases is to begin with a Howell design on $v_1+1$ symbols, and then remove all entries containing the superfluous symbol.  We thus define the symbol set $\tx_1$ and integer $\tv_1$ to be given by
$$
\tx_1 = \left\{
\begin{array}{ll} 
X_1, & \mbox{if } v_1 \mbox{ is even} \\
X_1 \cup \{\infty\}, & \mbox{if } v_1 \mbox{ is odd (where } \infty \notin X_1)
\end{array}
\right.
$$
and
$$
\tv_1 = \left\{
\begin{array}{ll} 
v_1, & \mbox{if } v_1 \mbox{ is even} \\
v_1+1, & \mbox{if } v_1 \mbox{ is odd}.
\end{array}
\right.
$$

We will show, that except for a finite number of exceptional vectors $\vv$, the upper bound on the packing number from Proposition~\ref{211-bound} can be achieved.  Let $m = \min \left\{ \binom{v_1}{2}, v_2 \lfloor v_1/2 \rfloor, v_2 v_3 \right\}$, and recall that Proposition~\ref{211-bound} states that $D(\vv,\kk,2) \leq m$.  It is easy to see that
$$
m = \left\{ 
\begin{array}{ll}
\binom{v_1}{2}, & \mbox{if } v_1 \leq v_2 \\
v_2 \lfloor v_1/2 \rfloor, & \mbox{if } v_2 < v_1 \leq 2v_3 \\
v_2 v_3, & \mbox{if } v_1 > 2v_3.
\end{array}
\right.
$$
Each of these three cases will be considered separately.  In addition, we will split the case in which $v_2 < v_1 \leq 2v_3$ into two separate cases: $v_2 < v_1 \leq 2v_2$ (in which we will employ a Howell design $H(v_2,\tv_1)$) and $2v_2 < v_1 \leq 2v_3$ (in which no such Howell design exists).

The first case which we consider is that $v_1 \leq v_2$. 

\begin{lemma} \label{binom_211}
Let $\vv=(v_1, v_2, v_3)$, where $v_1 \leq v_2 \leq v_3$, and $\kk=(2,1,1)$.  If $v_1 \notin \{3,4,5,6\}$, then $D(\vv,\kk,2) = \binom{v_1}{2}$.  
\end{lemma}

\begin{proof}
Proposition~\ref{211-bound} guarantees that $D(\vv,\kk,2) \leq \binom{v_1}{2}$, so it suffices to construct a packing of this size.  Construct a Room square $\rs(\tv_1-1)$ on symbol set $\tx_1$.  From this Room square, form an array $A$ by deleting the pair in any cell containing $\infty$ (if $v_1$ is odd), and appending $v_2-\tv_1+1$ empty rows and $v_3-\tv_1+1$ empty columns, so that $A$ is a $v_2 \times v_3$ array on symbol set $X_1$.  If $v_1$ is even, then $A$ contains $v_1-1$ nonempty rows, each with $v_1/2$ filled entries, while if $v_1$ is odd, then $A$ contains $v_1$ nonempty rows, each with $(v_1-1)/2$ filled entries.  In either case, we obtain a $2$-$(\vv,\kk,1)$ packing of size $\binom{v_1}{2}$.  
\end{proof}

The exceptional values $v_1 \in \{3,4,5,6\}$ arise from the fact that there is no $\mathrm{RS}(\tv_1)$ in these cases.  We deal with these exceptions now, showing that the bound of Proposition~\ref{211-bound} is met except for $\vv=(4,4,4)$ and $\vv=(5,5,5)$.

\begin{lemma} \label{lemma:333}
Let $\vv=(3,v_2,v_3)$, where $3 \leq v_2 \leq v_3$.  Then $D(\vv,\kk,2) = 3$.
\end{lemma}

\begin{proof}
By Proposition~\ref{211-bound}, we have that $D(\vv,\kk,2) \leq \binom{3}{2} = 3$.  A packing of size 3 is easy to construct: form a $v_2 \times v_3$ array, and place the three entries $\{1,2\}$, $\{1,3\}$ and $\{2,3\}$ so that no two occur in the same row or column.
\end{proof}

\begin{lemma} \label{lemma:444}
Let $\vv=(4,v_2,v_3)$, where $4 \leq v_2 \leq v_3$.  Then 
$$
D(\vv,\kk,2) = 
\left\{
\begin{array}{ll}
5, & \mbox{if } \vv=(4,4,4) \\
6, & \mbox{otherwise}.
\end{array}
\right.
$$
\end{lemma}

\begin{proof}
If $v_1=4$, then Proposition~\ref{211-bound} says that $D(\vv,\kk,2) \leq \binom{4}{2} = 6$.  If $\vv \neq (4,4,4)$, then we have that $v_3 \geq 5$.  In this case, it suffices to show a packing of size 6 with $\vv=(4,4,5)$, as otherwise, we may append $v_2-4$ empty rows and $v_3-5$ empty columns.  For such a packing, see Example~\ref{ex:445} in Appendix~\ref{section:211appendix}.  

The remaining case is that $\vv=(4,4,4)$.  A packing of size~6 would require a $4 \times 4$ array whose entries are pairs of elements in $\{1,2,3,4\}$, such that each pair occurs in the array, and no symbol appears twice in any row or column.  It is not difficult to show that no such array can exist.  However, it is possible to find a packing of size 5; an example is given in Appendix~\ref{section:211appendix}, Example~\ref{ex:444}.
\end{proof}

\begin{lemma} \label{lemma:555}
Let $\vv=(5,v_2,v_3)$, where $5 \leq v_2 \leq v_3$.  Then 
$$
D(\vv,\kk,2) = \left\{ 
\begin{array}{ll}
9, & \mbox{if } \vv=(5,5,5) \\
10, & \mbox{otherwise.}
\end{array}
\right.
$$
\end{lemma}

\begin{proof}
We know by Proposition~\ref{211-bound} that $D(\vv,\kk,2) \leq \binom{5}{2}=10$.  First consider the case that $\vv=(5,5,5)$. 
It is not difficult, although somewhat tedious, to show that no packing of size 10 exists; we leave it as an exercise for the reader.
An example of a packing of size 9 may be found in Appendix~\ref{section:211appendix}, Example~\ref{ex:555}.

Otherwise, we have that $v_3 \geq 6$, and it suffices to find a packing of size 10 where $\vv=(5,5,6)$.  An example may be found in Appendix~\ref{section:211appendix}, Example~\ref{ex:556}. 
\end{proof}

\begin{lemma} \label{lemma:666}
Let $\vv=(6,v_2,v_3)$, where $6 \leq v_2 \leq v_3$.  Then $D(\vv,\kk,2) = 15$.
\end{lemma}

\begin{proof}
From Proposition~\ref{211-bound}, we have that $D(\vv,\kk,2) \leq \binom{6}{2}=15$.  To show that this upper bound is realizable, it suffices to construct a $2$-$(\vv, \kk, 1)$ generalized packing of size 15 where $\vv=(6,6,6)$.  For such a packing, see Example~\ref{ex:666} of Appendix~\ref{section:211appendix}. 
\end{proof}

The next case we consider is that $v_2 < v_1 \leq 2v_2$.

\begin{lemma} \label{howell_211}
Let $\vv=(v_1,v_2,v_3)$, where $v_2 < v_1 \leq 2v_2$ and $v_2 \leq v_3$, and let $\kk=(2,1,1)$.  If $(v_1,v_2) \notin \{(3,2), (4,2), (4,3), (6,5), (7,5), (8,5)\}$, then $D(\vv,\kk,2) = v_2 \lfloor v_1/2 \rfloor$.
\end{lemma}

\begin{proof}
By Proposition~\ref{211-bound}, we know that $D(\vv,\kk,2) \leq v_2 \lfloor v_1/2 \rfloor$, and so we need only show the existence of a packing of this size.  Construct a Howell design $H(v_2,\tv_1)$, with symbols in $\tx_1$ and append $v_3-v_2$ empty columns.  If $v_1$ is odd, remove from this array the entries in any cell containing $\infty$, noting that there is one such cell in each row, so that the total number of filled positions is now $(\tv_1/2)-1 = \lfloor v_1/2 \rfloor$. The resulting $v_2 \times v_3$ array $A$ has entries in $X_1$, and the number of filled cells is $v_2 \lfloor v_1/2 \rfloor$, and so we have a $2$-$(\vv,\kk,1)$ packing of the desired size.
\end{proof}

The nonexistence of Howell designs $H(2,4)$, $H(3,4)$, $H(5,6)$ and $H(5,8)$ means that we have not yet constructed maximum packings in the following cases: $\vv=(3,2,v_3)$, where $v_3 \geq 2$; $\vv=(4,2,v_3)$, where $v_3 \geq 2$; $\vv=(4,3,v_3)$, where $v_3 \geq 3$; $\vv = (6,5,v_3)$, where $v_3 \geq 5$; $\vv=(7,5,v_3)$, where $v_3 \geq 5$; and $\vv=(8,5,v_3)$, where $v_3 \geq 5$.  The bound of Proposition~\ref{211-bound} projects maximum packings in these cases of sizes 2, 4, 6, 15, 15, and 20, respectively.  We deal with these cases in the following lemmas.
The first two find the packing number in the cases that $\vv=(3,2,v_3)$ or $\vv=(4,2,v_3)$; their proofs are straightforward, and so are omitted.

\begin{lemma}
Let $\vv=(3,2, v_3)$, where $v_3 \geq 2$.  Then $D(\vv,\kk,2)  = 2$.
\end{lemma}

\begin{lemma} 
Let $\vv=(4,2,v_3)$, where $v_3 \geq 2$.  Then 
$$
D(\vv,\kk,2) = 
\left\{
\begin{array}{ll}
2, & \mbox{\textrm if } v_3=2 \\
3, & \mbox{\textrm if } v_3=3 \\
4, & \mbox{\textrm if } v_3 \geq 4.
\end{array}
\right.
$$
\end{lemma}

\begin{lemma}
Let $\vv=(4,3,v_3)$, where $v_3 \geq 3$.  Then 
$$
D(\vv,\kk,2) = 
\left\{
\begin{array}{ll}
4, & \mbox{\textrm if } v_3=3,4 \\
5, & \mbox{if } v_3=5 \\
6, & \mbox{if } v_3 \geq 6.
\end{array}
\right.
$$
\end{lemma}

\begin{proof}
Proposition~\ref{211-bound} asserts that $D(\vv,\kk,2) \leq 6$.  If $v_3 \geq 6$, then it suffices to find a packing of size 6 in the case that $\vv=(4,3,6)$; such a packing can be found in Example~\ref{ex:436} of Appendix~\ref{section:211appendix}.

Now suppose that $v_3 =5$.  We must form a $3 \times 5$ array on an alphabet of size 4.  Note that each row can contain at most two entries, so six entries are only possible if each row contains exactly two non-empty cells; in this case, there must be a column which contains two non-empty cells.  However, if the $(i_1, j)$ and $(i_2,j)$ cells are both non-empty, then neither row $i_1$ nor row $i_2$ can contain two non-empty cells, as the only pair disjoint from the $(i_1,j)$-entry already appears in the $(i_2,j)$-entry, and vice-versa.  Thus, no packing of size 6 exists.  A packing of size 5 may be found in Appendix~\ref{section:211appendix}, Example~\ref{ex:435}. 

Similar arguments show that if $v_3 \leq 4$, then there can be no packing of size 5.  However, a packing of size 4 does exist if $\vv=(4,3,3)$ (see Appendix~\ref{section:211appendix}, Example~\ref{ex:433}); this array also forms a packing with $\vv=(4,3,4)$.
\end{proof}

\begin{lemma} \label{lemma:65v3}
Let $\vv=(6,5,v_3)$, where $v_3 \geq 5$.  Then 
$$
D(\vv,\kk,2) = 
\left\{
\begin{array}{ll}
13, & \mbox{if } v_3=5 \\
15, & \mbox{if } v_3 \geq 6.
\end{array}
\right.
$$
\end{lemma}

\begin{proof}
The upper bound asserted by Proposition~\ref{211-bound} is 15.  
If $v_3=5$, a packing of size 15 would be equivalent to a Howell design $H(5,6)$, which does not exist by Theorem~\ref{howell}.  A packing of size 13 is given in Appendix~\ref{section:211appendix}, Example~\ref{ex:655}.  Furthermore, as we now show, no packing of size 14 exists, as this would also imply the existence of an $H(5,6)$; this argument is due to Stinson (personal communication).

Suppose we have a packing of size 14, using the alphabet $\{1,2,\ldots,6\}$, and suppose without loss of generality that $\{5,6\}$ is the missing pair.  Now, as a $5\times 5$ array, we can assume (also without loss of generality) that the first four rows and columns each contain three pairs, and the final row and column each contain two pairs.  The symbols missing from the last row must be 5 and 6, and similarly the symbols missing from the last column must also be 5 and 6.  Now, if the cell in the bottom right corner is empty, then we can fill in the pair $\{5,6\}$ and we would obtain an $H(5,6)$. So we assume that this cell is already filled, and (without loss of generality) that it contains the pair $\{1,2\}$. Then the pair $\{3,4\}$ must already occur in the last row, and this pair must also occur in the last column. But the pair cannot occur twice, so we have a contradiction.

For $v_3 \geq 6$, it suffices to construct a packing of size 15 for $\vv=(6,5,6)$; an example of such a packing may be found in Appendix~\ref{section:211appendix}, Example~\ref{ex:656}.  
\end{proof}

\begin{lemma} \label{lemma:75v3}
Let $\vv=(7,5,v_3)$, where $v_3 \geq 5$.  Then $D(\vv,\kk,2) = 15$.
\end{lemma}

\begin{proof}
By Proposition~\ref{211-bound}, we have that $D(\vv,\kk,2) \leq 15$.  An example of a $2$-$(\vv,\kk,1)$ packing of size 15 with $\vv=(7,5,5)$ may be found in  Appendix~\ref{section:211appendix}, Example~\ref{ex:755}.
\end{proof}

\begin{lemma} \label{lemma:85v3}
Let $\vv=(8,5,v_3)$, where $v_3 \geq 5$, and $\kk=(2,1,1)$.  Then
$$
D(\vv,\kk,2) = 
\left\{
\begin{array}{ll}
19, & \mbox{if } v_3=5 \\
20, & \mbox{if } v_3\geq 6.
\end{array}
\right.
$$
\end{lemma}

\begin{proof}
Proposition~\ref{211-bound} gives us an upper bound of 20.  If $v_3=5$, however, then there is no packing of size 20, as otherwise, there would exist an $H(8,5)$.  A packing of size 19 is given in Appendix~\ref{section:211appendix}, Example~\ref{ex:855}.

If $v_3 \geq 6$, then it suffices to find a packing of size 20 for $\vv=(8,5,6)$; an example may be found in Appendix~\ref{section:211appendix}, Example~\ref{ex:856}. 
\end{proof}

Having dealt with those exceptions, we move on to consider the case where $2v_2 < v_1 \leq 2v_3$.

\begin{lemma} \label{soma_211}
Let $\vv=(v_1,v_2,v_3)$, where $2v_2 < v_1 \leq 2v_3$.  Then $D(\vv,\kk,2) = v_2 \lfloor v_1/2 \rfloor$.
\end{lemma}

\begin{proof}
Proposition~\ref{211-bound} asserts that $D(\vv,\kk,2) \leq v_2 \lfloor v_1/2 \rfloor$, and so it suffices to find a packing of this size.  

Note that the condition $2v_2 < v_1$ implies that $v_1 \geq 4$.  Let us first suppose that $v_1 \neq 4$.  Then there exists a $\soma(2,\tv_1/2)$ with symbol set $\tx_1$.  Let $A$ be the array formed by taking the first $v_2$ rows of the $\soma$ and appending $v_3-v_1/2$ empty columns.  Note that each row of $A$ contains $\lceil v_1/2 \rceil$ nonempty cells.  Now, if $v_1$ is odd, delete from $A$ the entries in any cell containing $\infty$.  We obtain a $v_2 \times v_3$ array on symbol set $X_1$, with $v_2 \lfloor v_1/2 \rfloor$ nonempty cells, which gives the desired $2$-$(\vv,\kk,1)$ packing.

If $v_1=4$, then no $\mathrm{SOMA}(2,\tv_1/2)$ exists.  The condition $2v_2 < v_1 \leq 2v_3$ implies that $v_2=1$ and $v_3 \geq 2$.  We seek a $2$-$(\vv,\kk,1)$ packing of size 2, which is trivial to find.
\end{proof}

The final case is that $v_1 > 2v_3$.  

\begin{lemma}
Let $\vv=(v_1,v_2,v_3)$, where $v_1 > 2v_3$ and $v_2 \leq v_3$.  If $v_3 \neq 2$, then $D(\vv,\kk,2) = v_2 v_3$.
\end{lemma}

\begin{proof}
By Proposition~\ref{211-bound}, we know that $D(\vv,\kk,2) \leq v_2 v_3$, and so it suffices to show that a packing of this size exists.  Since $v_3 \neq 2$, we may construct a $\soma(2,v_3)$, and let $A$ be the $v_2 \times v_3$ array formed by taking the first $v_2$ rows of the $\soma$.  Note that each of the $v_2 v_3$ cells of $A$ are filled.  The array $A$ contains $2v_3$ distinct symbols, and since $v_1 > 2v_3$, it follows that $A$ is a $2$-$(\vv,\kk,1)$ packing.
\end{proof}

The only exception is the case that $v_3=2$, as there is no $\soma(2,2)$.  Since $v_1 > 2v_3$ and $v_2 \leq v_3$, we have in this case that $v_1 \geq 5$ and $v_2 \leq 2$.

\begin{lemma}
Let $\vv=(v_1, v_2, 2)$, where $v_1 \geq 5$ and $v_2 \leq 2$.  Then
$$
D(\vv,\kk,2) = \left\{
\begin{array}{ll}
2, & \mbox{if } v_2=1 \\
3, & \mbox{if } v_1=5 \mbox{ and } v_2=2 \\
4, & \mbox{otherwise}.
\end{array}
\right.
$$
\end{lemma}

\begin{proof}
By Proposition~\ref{211-bound}, we have that $D(\vv,\kk,2) \leq 2$ if $v_2=1$ and $D(\vv,\kk,2) \leq 4$ if $v_2 = 2$.  If $v_2=1$, then a packing of size 2 is easy to construct.  If $\vv=(5,2,2)$, then it is easy to see that no packing of size 4 exists; a packing of size 3 can be obtained by deleting the $(2,2)$-entry of the array given in Example~\ref{ex:622} of Appendix~\ref{section:211appendix}.  To see that $D(\vv,\kk,2) = 4$ in the remaining case, it suffices to show that $D(\vv,\kk,2)=4$ for $\vv=(6,2,2)$; a $2$-$(\vv,\kk,1)$ packing of size 4 may also be found in Example~\ref{ex:622}.
\end{proof}

We conclude this section by combining the results for $\kk=(2,1,1)$ above into the following theorem.

\begin{theorem} \label{thm:2_1_1_summary}
Suppose $\vv=(v_1,v_2,v_3)$ and $\kk=(2,1,1)$, where $v_2\leq v_3$.  Then, except for the specific values listed in Table~\ref{table:2_1_1_exceptions} below, there exists a $2$-$(\vv,\kk,1)$ generalized packing meeting the bound of Proposition~\ref{211-bound}.
\end{theorem}

\begin{table}[hbt]
\centering
\renewcommand{\arraystretch}{1.2}
\begin{tabular}{|c|c|c|} \hline
$\vv=(v_1,v_2,v_3)$  & Projected bound & Packing number \\ \hline
$(4,2,2)$            & 4               & 2 \\ 
$(4,2,3)$            & 4               & 3 \\
$(4,3,3)$, $(4,3,4)$ & 6               & 4 \\
$(4,3,5)$            & 6               & 5 \\
$(4,4,4)$            & 6               & 5 \\ \hline
$(5,2,2)$            & 4               & 3 \\ 
$(5,5,5)$            & 10              & 9 \\ \hline
$(6,5,5)$            & 15              & 13 \\ \hline
$(8,5,5)$            & 20              & 19 \\
\hline
\end{tabular}
\renewcommand{\arraystretch}{1.0}
\caption{Exceptions for when $D(\vv,(2,1,1),2)$ does not meet the bound of Proposition~\ref{211-bound}.}
\label{table:2_1_1_exceptions}
\end{table}

\subsection{$\kk = (2,2)$}

This is the other case for $k=4$ where no generalized $2$-$(\vv,\kk,1)$ design can exist (except in the trivial case $\vv=(2,2)$).  However, once again there are objects in the literature which arise in this situation.  The graphical interpretation of Section~\ref{section:graphical} has a straightforward interpretation.  A $2$-$(\vv,\kk,1)$ generalized packing with $\vv=(v_1,v_2)$ and $\kk=(2,2)$ is equivalent to a packing of 4-cycles (i.e.\ $K_{2,2}$) into a complete bipartite graph $K_{v_1,v_2}$ with the following additional constraint: if $(a,x,b,y)$ and $(a,u,b,v)$ are 4-cycles in the packing, then $\{x,y\}=\{u,v\}$.  That is, if $K_{v_1,v_2}$ has vertex set $X_1\dot\cup X_2$, then no pair of vertices in $X_1$ is ever repeated, and likewise for $X_2$.

In the case where $K_{v_1,v_2}$ has a decomposition into 4-cycles with this property, the decomposition is said to be {\em monogamous}.  Such decompositions were introduced in the 1999 paper of Lindner and Rosa~\cite{lindner_rosa}, where they proved the following.

\begin{theorem}[Lindner and Rosa%
~{\cite[Theorem 2.7]{lindner_rosa}}%
] \label{thm:lindner_rosa}
A complete bipartite graph $K_{v_1,v_2}$ admits a monogamous decomposition into 4-cycles if and only if $v_1$ and $v_2$ are both even, $v_1\leq v_2 \leq 2v_1-2$, and either $v_1=v_2=2$ or $v_1,v_2\geq 6$.
\end{theorem}

The first two necessary conditions (i.e.\ that $v_1$ and $v_2$ are both even, and $v_1\leq v_2 \leq 2v_1-2$) are straightforward.  However, we note that there is no monogamous 4-cycle decomposition of $K_{4,4}$ or $K_{4,6}$.

In almost all cases, Lindner and Rosa's proof used Howell designs in the construction of monogamous 4-cycle decompositions, by means of another, related 
object known as a {\em symmetric Howell square}.  In what follows, we extend their results to the more general case of packings.  Our approach involves applying Proposition~\ref{merging} to a generalized packing with $\kk=(2,1,1)$, and merging the two parts with $k_i=1$ to form a packing with $\kk=(2,2)$.  In fact, in many cases if we begin with a maximum packing with $\kk=(2,1,1)$, this approach yields a maximum packing with $\kk=(2,2)$.  We note that, by using maximum packings with $\kk=(2,1,1)$, our approach is also based on Howell designs.

When $\vv=(v_1,v_2)$ and $\kk=(2,2)$, we can assume without loss of generality that $v_1 \leq v_2$.  In this case, the bound of Proposition~\ref{t=2_bound} gives us the following.
\begin{proposition} \label{22-bound}
Let $\vv=(v_1, v_2)$, where $v_1 \leq v_2$, and $\kk=(2,2)$.  Then
$$
D(\vv,\kk,2) \leq \min \left\{ \binom{v_1}{2}, \left\lfloor \frac{v_1}{2} \left\lfloor \frac{v_2}{2} \right\rfloor \right\rfloor, \left\lfloor \frac{v_2}{2} \left\lfloor \frac{v_1}{2} \right\rfloor \right\rfloor \right\}
$$
\end{proposition}

The exact value of this upper bound, which is given in Table~\ref{table:22-bound}, depends on the parities of $v_1$ and $v_2$, as well as their relative sizes.  
\renewcommand{\arraystretch}{1.2}
\begin{table}[hbt]
\centering
\begin{tabular}{|c|c|c|c|} \hline
$v_1$ & $v_2$ & Range                      & Upper bound \\ \hline
\multirow{2}{*}{Even} & \multirow{2}{*}{Even} 
              & $v_2 \geq 2v_1$            & $\binom{v_1}{2}$                \\
      &       & $v_1 \leq v_2 < 2v_1$      & $v_1 v_2 / 4$                   \\ \hline
\multirow{2}{*}{Even} & \multirow{2}{*}{Odd} 
              & $v_2 > 2v_1$               & $\binom{v_1}{2}$                \\ 
      &       & $v_1 < v_2 < 2v_1$         & $v_1 (v_2-1)/4$                 \\ \hline
\multirow{2}{*}{Odd} & \multirow{2}{*}{Even} 
              & $v_2 \geq 2v_1$            & $\binom{v_1}{2}$                \\
      &       & $v_1 < v_2 < 2v_1$         & $v_2 (v_1-1)/4$                 \\ \hline
\multirow{2}{*}{Odd} & \multirow{2}{*}{Odd} 
              & $v_2 > 2v_1$               & $\binom{v_1}{2}$                \\
      &       & $v_1 \leq v_2 < 2v_1$      & $\lfloor v_2 (v_1-1)/4 \rfloor$ \\ \hline
\end{tabular}
\caption{Upper bounds on $D(\vv,\kk,2)$ when $\kk=(2,2)$.}
\label{table:22-bound}
\end{table}
\renewcommand{\arraystretch}{1.0}
It is clear from Table~\ref{table:22-bound} that the upper bound posited by Proposition~\ref{22-bound} is $\binom{v_1}{2}$ whenever $v_2 \geq 2v_1$.  We first consider this case.

\begin{lemma} \label{binom_22}
Suppose that $v_2 \geq 2v_1 \geq 4$.  Let $\vv=(v_1, v_2)$ and $\kk=(2,2)$.  Then $D(\vv,\kk,2) = \binom{v_1}{2}$, except possibly if $\vv \in \{(4,8),(5,10)\}$.
\end{lemma}

\begin{proof}
Proposition~\ref{22-bound} guarantees that $D(\vv,\kk,2) \leq \binom{v_1}{2}$, and so it suffices to prove the existence of a packing of this size.  Let $x_1 = v_1$, $x_2 = \lfloor v_2/2 \rfloor$ and $x_3 = \lceil v_2/2 \rceil$.  Notice that $x_1 \leq x_2 \leq x_3$ and $x_2 + x_3 = v_2$.  Let $\xx=(x_1, x_2, x_3)$ and $\kkappa = (2,1,1)$.  By Lemmas~\ref{binom_211} to~\ref{lemma:666}, there exists a $2$-$(\xx,\kkappa,1)$ packing of size $\binom{x_1}{2} = \binom{v_1}{2}$.  Hence, by Proposition~\ref{merging}, there is a $2$-$(\vv,\kk,1)$ packing of size $\binom{v_1}{2}$.  
\end{proof}

The exceptions for $\vv=(4,8)$ and $(5,10)$ arise from the fact that there is no $2$-$(\xx, (2,1,1),1)$ packing meeting the bound of Proposition~\ref{211-bound} if $\xx=(4,4,4)$ or $(5,5,5)$.  However, in both of these exceptional cases, we can obtain the exact value of $D(\vv,\kk,2)$ as follows.

\begin{lemma} \label{lemma:4_8}
Let $\vv=(4,8)$ and $\kk=(2,2)$.  Then $D(\vv,\kk,2) = 5$.
\end{lemma}

\begin{proof}
Since $D((4,4,4), (2,1,1),2) = 5$ by Lemma~\ref{lemma:444}, then by merging parts we have that $D((\vv,\kk,2) \geq 5$.  We shall show by contradiction that there can be no packing of size~6.

Suppose the symbol set is $X_1\cup X_2 = \{1,2,3,4\}\cup\{a,b,c,d,e,f,g,h\}$.  Now, any packing of size~6 would have to include all pairs chosen from $\{1,2,3,4\}$; in particular, it would have to use the pairs $\{1,2\}$, $\{1,3\}$ and $\{1,4\}$.  Without loss of generality, we can assume that the blocks containing these pairs are $(\{1,2\},\{a,b\})$, $(\{1,3\},\{c,d\})$ and $\{1,4\},\{e,f\})$.  The remaining blocks must contain the pairs $\{2,3\}$, $\{2,4\}$ and $\{3,4\}$: each block must also contain a pair containing at least one of the symbols $g,h$.  However, at most two such blocks can be formed.
\end{proof}

\begin{lemma} \label{lemma:5_10}
Let $\vv=(5,10)$ and $\kk=(2,2)$.  Then $D(\vv,\kk,2) = 10$.
\end{lemma}

\begin{proof}
We can obtain a packing of size~10 by deleting a point from the monogamous 4-cycle decomposition of $K_{6,10}$, given by Lindner and Rosa~\cite[Example 2.4]{lindner_rosa}, to obtain the following:
\[ 
\begin{array}{c}
(\{1,2\},\{a,b\}) \\
(\{1,3\},\{c,d\}) \\
(\{1,4\},\{e,f\}) \\
(\{1,5\},\{g,h\}) \\
(\{2,3\},\{e,i\})
\end{array}
\qquad
\begin{array}{c}
(\{2,4\},\{c,g\}) \\
(\{2,5\},\{d,j\}) \\
(\{3,4\},\{h,j\}) \\
(\{3,5\},\{a,f\}) \\
(\{4,5\},\{b,i\}).
\end{array}
\]
\end{proof}

It remains to consider the case that $v_1 \leq v_2 < 2v_1$.  Here, the upper bound from Lemma~\ref{22-bound} is $\lfloor v_1/2 \rfloor \cdot \lfloor v_2/2 \rfloor$ unless $v_1$ and $v_2$ are both odd.  

\begin{lemma} \label{22-howell}
Suppose that $2 \leq v_1 \leq v_2 < 2v_1$.  Let $\vv=(v_1, v_2)$ and $\kk=(2,2)$.  Then there is a $2$-$(\vv,\kk,1)$ packing of size $\lfloor v_1/2 \rfloor \cdot \lfloor v_2/2 \rfloor$, except if $\vv \in \{(4,4), (4,6)\}$, and except possibly in the case that $\vv \in \{(4,5), (4,7)\}$.  
\end{lemma}

\begin{proof}
As in the proof of Lemma~\ref{binom_22}, we let $x_1 = v_1$, $x_2 = \lfloor v_2/2 \rfloor$, $x_3 = \lceil v_2/2 \rceil$, $\xx=(x_1, x_2, x_3)$ and $\kkappa = (2,1,1)$.  It is not difficult to show that $x_2 + 1 \leq x_1 \leq 2x_2$ if $x_2$ is even, and $x_2 + 1 \leq x_1 \leq 2x_2+1$ if $x_2$ is odd.  Hence by Lemma~\ref{soma_211} (in the case that $x_1=2x_2+1$, noting that here $2x_2<x_1 \leq 2x_3$) and Lemmas~\ref{howell_211} and~\ref{lemma:65v3} to~\ref{lemma:85v3} (otherwise), there exists a $2$-$(\xx,\kkappa,1)$ packing of size $x_2 \lfloor x_1/2 \rfloor = \lfloor v_1/2 \rfloor \cdot \lfloor v_2/2 \rfloor$.  Applying Proposition~\ref{merging} gives a $2$-$(\vv,\kk,1)$ packing of the same size.

The results quoted above leave open the cases where $\vv \in \{(4,4), (4,5), (4,6), (4,7), (6,10), (8,10)\}$.  In the cases where $v_1,v_2$ are both even, each of these correspond to monogamous 4-cycle decompositions of $K_{v_1,v_2}$, so we can appeal to Theorem~\ref{thm:lindner_rosa} above (due to Lindner and Rosa).  This shows that decompositions (and thus packings of size $\lfloor v_1/2 \rfloor \cdot \lfloor v_2/2 \rfloor$) do not exist for $\vv=(4,4)$ and $\vv=(4,6)$, but do exist for $\vv=(6,10)$ and $\vv=(8,10)$.  Lindner and Rosa~\cite[Examples 2.4, 3.2, 3.3]{lindner_rosa} give examples of such decompositions.
\end{proof}

\begin{corollary}
Suppose that $2 \leq v_1 \leq v_2 < 2v_1$, and let $\vv=(v_1, v_2)$ and $\kk=(2,2)$.  If $\vv \notin \{ (4,4), (4,5), (4,6), (4,7)\}$ and at least one of $v_1$ and $v_2$ is even, then $D(\vv,\kk,2) = \lfloor v_1/2 \rfloor \cdot \lfloor v_2/2 \rfloor$.  
\end{corollary}

\begin{proof}
In the case that $v_1 \leq v_2 < 2v_1$, where $v_1$ and $v_2$ are not both odd, the upper bound on $D(\vv,\kk,2)$ given by Proposition~\ref{22-bound} simplifies to $\lfloor v_1/2 \rfloor \cdot \lfloor v_2/2 \rfloor$.  Lemma~\ref{22-howell} guarantees the existence of a packing of this size.  
\end{proof}

The exceptions in Lemma~\ref{22-howell} arise from instances in which no Howell design exists, and as a result there does not exist a $2$-$(\xx,\kkappa,1)$ packing meeting the bound of Proposition~\ref{211-bound}.  The values of $D(\xx,\kkappa,2)$ for the vectors $\xx = (4,2,2)$, $(4,2,3)$, $(4,3,3)$, $(4,3,4)$, are 2, 3, 4 and 4, respectively.  Using Proposition~\ref{merging}, we can construct $2$-$(\vv,\kk,1)$ packings of these sizes for $\vv=(4,4)$, $(4,5)$, $(4,6)$ and $(4,7)$, respectively.  Applying the upper bound given by Proposition~\ref{22-bound}, and noting the non-existence of monogamous 4-cycle decompositions of $K_{4,4}$ and $K_{4,6}$, we obtain the following:
$$
\begin{array}{l}
2 \leq D((4,4),\kk,2) \leq 3 \\
3 \leq D((4,5),\kk,2) \leq 4 \\
4 \leq D((4,6),\kk,2) \leq 5 \\
4 \leq D((4,7),\kk,2) \leq 6.  
\end{array}
$$
In each of these cases, we can determine the packing numbers exactly.

\begin{lemma}
Let $\vv=(4,4)$ and $\kk=(2,2)$.  Then $D(\vv,\kk,2) = 2$.
\end{lemma}

\begin{proof}
Suppose we have a maximum $2$-$(\vv,\kk,1)$ generalized packing.  We can assume that it contains the block $(\{1,2\},\{a,b\})$.  If the packing contains the block $(\{3,4\},\{c,d\})$, then it can contain no other blocks, and so has size 2.  Otherwise, we can assume that a second block in the packing is $(\{1,3\},\{c,d\})$; again, no other block can be added.
\end{proof}

\begin{lemma}
Let $\vv=(4,5)$ and $\kk=(2,2)$.  Then $D(\vv,\kk,2) = 3$.
\end{lemma}

\begin{proof}
We know that $D(\vv,\kk,2) \geq 3$.  Suppose we have a maximum $2$-$(\vv,\kk,1)$ packing.  Without loss of generality, we can assume that it contains the block $(\{1,2\}, \{a,b\})$.  If the packing contains a block $(B_1, B_2)$ such that $B_1 \cap \{1,2\} = \emptyset$ and $B_2 \cap \{a,b\} = \emptyset$ (we can assume that $(B_1, B_2) = (\{3,4\},\{c,d\})$), then it can contain no other block, and so this packing, of size 2, cannot be maximum. 

So each block $(B_1, B_2)$ in the packing must have either $B_1 \cap \{1,2\} \neq \emptyset$ or $B_2 \cap \{a,b\} \neq \emptyset$.  If there is a block $(B_1,B_2)$ such that $B_1 \cap \{1,2\} \neq \emptyset$, we can assume that this block is $(\{1,3\}, \{c,d\})$.  Now the only remaining possible blocks are $(\{2,4\},\{c,e\})$, $(\{2,4\},\{d,e\})$, $(\{3,4\},\{a,e\})$ and $(\{3,4\},\{b,e\})$; however, the packing can contain at most one of these blocks.

Similarly, if there is a block $(B_1,B_2)$ such that $B_2 \cap \{a,b\} \neq \emptyset$, then there can be at most one further block added.
\end{proof}

\begin{lemma}
Let $\vv=(4,6)$ and $\kk=(2,2)$.  Then $D(\vv,\kk,2) = 4$.
\end{lemma}

\begin{proof}
Since there exists a $2$-$((4,3,3),(2,1,1),1)$ packing of size 4 (see Example~\ref{ex:433} in Appendix~\ref{section:211appendix}), we have that $D(\vv,\kk,2) \geq 4$ by Proposition~\ref{merging}.

To see that there can be no larger packing, it is enough to show that in any packing of size at least 4, no element of $X_1$ can occur in three blocks.  Let $X_1 = \{1,2,3,4\}$ and $X_2=\{a,b,c,d,e,f\}$.  Suppose that there are three blocks containing element $1 \in X_1$.  Since no pair in $X_1, X_2$ or $X_1 \times X_2$ can be repeated in a block, we can assume without loss of generality that these blocks are $(\{1,2\},\{a,b\})$, $(\{1,3\},\{c,d\})$ and $(\{1,4\},\{e,f\})$.  Now consider the pair $\{2,3\} \in X_1$.  The only possible elements of $X_2$ which can occur in a block with both 2 and 3 are $e$ and $f$, but the pair $\{e,f\}$ has already been used.  So no block can contain $\{2,3\}$.  Similarly, no block can contain $\{2,4\}$ or $\{3,4\}$, so the packing has size 3.
\end{proof}

\begin{lemma}
Let $\vv=(4,7)$ and $\kk=(2,2)$.  Then $D(\vv,\kk,2) = 4$.
\end{lemma}

\begin{proof}
Since there is a $2$-$((4,3,4),(2,1,1),1)$ packing of size 4 (given by the array in Example~\ref{ex:433} with an empty column added), we have that $D(\vv,\kk,2) \geq 4$.

Consider a maximum $2$-$(\vv,\kk,1)$ packing on symbol set $X_1 \cup X_2 = \{1,2,3,4\} \cup \{a,b,c,d,e,f,g\}$.  If there is no element of $X_1$ occurring in three blocks, then the packing can have size at most 4.  Otherwise, we can assume without loss of generality that the packing contains blocks $(\{1,2\},\{a,b\})$, $(\{1,3\},\{c,d\})$, $(\{1,4\},\{e,f\})$.  It is a simple exercise to show that only one further block can be added. 
\end{proof}

Finally, we discuss the case in which $v_1$ and $v_2$ are both odd and $v_1 \leq v_2 < 2v_1$.  In this case, Proposition~\ref{22-bound} tells us that $D(\vv,\kk,2) \leq \lfloor v_2(v_1-1)/4 \rfloor$, while Lemma~\ref{22-howell} proves the existence of a packing of size $\lfloor v_1/2 \rfloor \cdot \lfloor v_2/2 \rfloor$.  Unfortunately, the construction of Lemma~\ref{22-howell} does not always give an optimal packing.  For instance, if $\vv=(5,5)$ and $\kk=(2,2)$, then the packing found by Lemma~\ref{22-howell} has size 4.  However, the following blocks give a maximum packing, of size $5 = v_2(v_1-1)/4$:
$$
\begin{array}{l}
(\{1,2\}, \{a,b\}) \\
(\{1,3\}, \{c,d\}) \\
(\{2,4\}, \{c,e\}) \\
(\{3,5\}, \{a,e\}) \\
(\{4,5\}, \{b,d\}).
\end{array}
$$

\begin{problem}
Let $\vv=(v_1,v_2)$, where $v_1$ and $v_2$ are odd and $v_1 \leq v_2 \leq 2v_1-1$.  Determine whether $D(\vv,\kk,2)$ meets the bound of Proposition~\ref{22-bound}.
\end{problem}

We remark that in the situation above, we have that 
\[ \left\lfloor \frac{v_1}{2} \right\rfloor  \cdot \left\lfloor \frac{v_2}{2} \right\rfloor = \frac{(v_1-1)(v_2-1)}{4} \leq D(\vv,\kk,2) \leq \left\lfloor \frac{v_2(v_1-1)}{4} \right\rfloor, \]
so the upper and lower bounds are quite close.


We conclude this section by combining the results for $\kk=(2,2)$ above into the following theorem.

\begin{theorem} \label{thm:2_2_summary}
Suppose $\vv=(v_1,v_2)$ and $\kk=(2,2)$, where $2\leq v_1\leq v_2$.  Then there exists a $2$-$(\vv,\kk,1)$ generalized packing meeting the bound of Proposition~\ref{22-bound}, except for the specific values listed in Table~\ref{table:2_2_exceptions} below, 
and with possible exception of the case where $v_1$, $v_2$ are both odd and $v_1 \leq v_2 \leq 2v_1-1$.
\end{theorem}
\begin{table}[hbt]
\centering
\renewcommand{\arraystretch}{1.2}
\begin{tabular}{|c|c|c|} \hline
$\vv=(v_1,v_2)$  & Projected bound & Packing number \\ \hline
$(4,4)$ & 4 & 2 \\
$(4,5)$ & 4 & 3 \\
$(4,6)$ & 6 & 4 \\
$(4,7)$ & 6 & 4 \\
$(4,8)$ & 6 & 5 \\
\hline
\end{tabular}
\renewcommand{\arraystretch}{1.0}
\caption{Known exceptions for when $D(\vv,(2,2),2)$ does not meet the bound of Proposition~\ref{22-bound}.}
\label{table:2_2_exceptions}
\end{table}
%
%
%

\subsection{$\kk = (1,1,1,1)$} \label{section:1111}

We recall from Section~\ref{section:PA} that if $\kk=(1,1,1,1)$ and $\vv=(s,s,s,s)$, a $2$-$(\vv,\kk,1)$ generalized packing with $N$ blocks is equivalent to a packing array $\mathrm{PA}(N;4,s,2)$.  When $s \notin\{2,6\}$, the existence of two mutually orthogonal Latin squares of order $s$ implies the existence of a maximum generalized packing of size $s^2$.  For other vectors $\vv=(v_1,v_2,v_3,v_4)$ (where we assume, without loss of generality, that $v_1 \leq v_2 \leq v_3 \leq v_4$), we can appeal to Proposition~\ref{max_PA_construction} to obtain a maximum generalized packing, provided there exist a pair of $v_1\times v_2$ orthogonal Latin rectangles.

If $v_2 = 2$ or 6, then there do not exist two $\mathrm{MOLS}(v_2)$.  In particular, this means that there does not exist a $2$-$(\vv,\kk,1)$ packing of size $v_1 v_2$ if $\vv=(2,2,2,2)$ or $(6,6,6,6)$.  However, for certain values of $\vv$ with $v_2 \in \{2,6\}$, we can still obtain a packing of size $v_1 v_2$.  The following lemmas give us the exact values of the packing number in the remaining cases.

\begin{lemma}
Let $\vv=(v_1, v_2, v_3, v_4)$, where $v_2 = 2$ and $v_1 \leq v_2 \leq v_3 \leq v_4$, and let $\kk=(1,1,1,1)$.  Then
$$
D(\vv,\kk,2) = \left\{
\begin{array}{ll}
2, & \mbox{\textrm if } \vv=(1,2,v_3,v_4) \mbox{\textrm\ where } 3 \leq v_3 \leq v_4 \\
2, & \mbox{\textrm if } \vv=(2,2,2,2) \\
3, & \mbox{\textrm if } \vv=(2,2,2,3) \\
4, & \mbox{\textrm if } \vv=(2,2,2,v_4) \mbox{\textrm\ where } v_4 \geq 4 \\
4  & \mbox{\textrm if } \vv=(2,2,v_3,v_4) \mbox{\textrm\ where } 3 \leq v_3 \leq v_4.
\end{array}
\right.
$$
\end{lemma}

\begin{proof}
If $v_1=1$ or $\vv=(2,2,2,2)$, the result is straightforward to verify.  If $\vv=(2,2,2,3)$, it is not hard to show by contradiction that there is no packing of size 4. However, the following two arrays give a packing of size 3:
$$
\begin{array}{ccc}
\begin{array}{cc} 
1 & 2 \\ 
2 &  \\ 
\end{array}
& 
\mbox{ and}
& 
\begin{array}{cc} 
1 & 2 \\ 
3 &  \\ 
\end{array}
\end{array}\; .
$$ 

Next, suppose that $\vv=(2,2,2,v_4)$, where $v_4\geq 4$.  The arrays
$$
\begin{array}{ccc}
\begin{array}{cc} 
1 & 2 \\ 
2 & 1 \\ 
\end{array}
& 
\mbox{ and }
& 
\begin{array}{cc} 
1 & 2 \\ 
3 & 4 \\ 
\end{array}
\end{array}
$$ 
give a maximum packing.

Finally, if $\vv=(2,2,v_3,v_4)$, where $3 \leq v_3 \leq v_4$, then the following two arrays give a maximum packing:
$$
\begin{array}{ccc}
\begin{array}{cc} 
1 & 2 \\ 
2 & 3 \\ 
\end{array}
& 
\mbox{ and }
& 
\begin{array}{cc} 
1 & 2 \\ 
3 & 1 \\ 
\end{array}
\end{array}\; .
$$ 
\end{proof}

The other exceptions arise as the result of the non-existence of two MOLS of order~6.

\begin{lemma}
Let $\vv=(v_1, v_2, v_3, v_4)$, where $v_2 = 6$ and $v_1 \leq v_2 \leq v_3 \leq v_4$, and let $\kk=(1,1,1,1)$.  Then
$$
D(\vv,\kk,2) = \left\{
\begin{array}{ll}
6v_1 & \mbox{\textrm if } v_1 \leq 5 \\
34, & \mbox{\textrm if } \vv=(6,6,6,6) \\
36, & \textrm{otherwise.}
\end{array}
\right.
$$
\end{lemma}

\begin{proof}
Proposition~\ref{t=2_bound} tells us that $D(\vv,\kk,2) \leq 6v_1$.  
If $\vv=(6,6,6,6)$, then a $2$-$(\vv,\kk,1)$ design with $N$ blocks is equivalent to a pair of $6\times 6$ mutually orthogonal partial Latin squares with $N$ entries: it is known that the largest possible number of entries is 34 (see Abdel-Ghaffar~\cite{abdelghaffar}).
We next consider the case in which $v_1 \leq 5$.  Although there do not exist any $\mathrm{MOLS}(6)$, there does exist a pair of orthogonal $5 \times 6$ Latin rectangles on 6 symbols (obtained by filling in the $(5,5)$ and $(5,6)$ entries in~\cite[Example III.4.3]{handbook}).

In the remaining cases, we have that $v_4 \geq 7$.  In the Problem Session at the CanaDAM conference in June 2011, the authors posed the problem of determining $D(\vv,\kk,2)$ where $\vv=(6,6,6,7)$ and $\kk=(1,1,1,1)$ as a programming challenge: by the following morning two separate solutions had been provided (the first by C.~Sato and M.~Silva, the second by T.~Britz) which showed that $D(\vv,\kk,2)=36$.

The following two arrays (obtained by Sato and Silva) give a packing of size 36:
$$
\begin{array}{ccc}
\begin{array}{cccccc}
5 & 3 & 4 & 2 & 6 & 1 \\
6 & 4 & 1 & 5 & 2 & 3 \\
2 & 5 & 3 & 1 & 4 & 6 \\
3 & 2 & 5 & 6 & 1 & 4 \\
4 & 1 & 6 & 3 & 5 & 2 \\
1 & 6 & 2 & 4 & 3 & 5 
\end{array}
&
\mbox{ and }
&
\begin{array}{cccccc}
3 & 2 & 5 & 1 & 7 & 4 \\
4 & 7 & 3 & 2 & 5 & 1 \\
7 & 6 & 4 & 5 & 1 & 2 \\
5 & 4 & 1 & 3 & 2 & 6 \\
2 & 1 & 6 & 7 & 4 & 3 \\
6 & 5 & 2 & 4 & 3 & 7
\end{array}
\end{array} \; .
$$
Alternatively, the following solution was found by Britz:
$$
\begin{array}{ccc}
\begin{array}{cccccc}
1 & 2 & 3 & 4 & 5 & 6 \\
2 & 1 & 4 & 3 & 6 & 5 \\
3 & 4 & 5 & 6 & 1 & 2 \\
4 & 5 & 6 & 2 & 3 & 1 \\
5 & 6 & 2 & 1 & 4 & 3 \\
6 & 3 & 1 & 5 & 2 & 4 
\end{array}
&
\mbox{ and }
&
\begin{array}{cccccc}
1 & 2 & 3 & 4 & 5 & 6 \\
3 & 4 & 1 & 5 & 2 & 7 \\
4 & 7 & 2 & 1 & 6 & 5 \\
5 & 1 & 4 & 6 & 7 & 3 \\
6 & 5 & 7 & 2 & 3 & 1 \\
7 & 6 & 5 & 3 & 1 & 2
\end{array}
\end{array} \; .
$$
\end{proof}

\section{Conclusion}
We have seen that, for specific values of the parameters of a $t$-$(\vv,\kk,\lambda)$ generalized packing, examples often correspond to other interesting combinatorial objects.  It seems possible that other classes of combinatorial designs may arise as instances of generalized packings, which warrants further investigation.  

When $t=2$, $\lambda=1$ and $k=3$ or 4, we have determined the generalized packing numbers exactly, except in the case that $\kk=(2,2)$ and the entries of $\vv$ are both odd.  In particular, in the remaining cases, $D(\vv,\kk,2)$ has been shown to meet the upper bound given by Proposition~\ref{t=2_bound} with only a finite number of exceptional values of $\vv$.  We suspect that this may hold more widely, yet demonstrating this will likely be extremely challenging.  As an example of the difficulty of this problem, a special case would be the determination of $D(\vv,\kk,2)$ for $\vv=(10,10,10,10,10)$ and $\kk=(1,1,1,1,1)$; meeting the bound of Proposition~\ref{t=2_bound} would require proving the existence of three $\mathrm{MOLS}$ of order 10.

\subsection*{Acknowledgements}
The authors acknowledge support from a PIMS Postdoctoral Fellowship (R.~F.~Bailey) and an NSERC Postdoctoral Fellowship (A.~C.~Burgess).  We would like to thank Daniel Horsley for communicating the results in~\cite{CHW_trails,colbourn_horsley_wang}, Liz Billington and Nick Cavenagh for communicating~\cite{lindner_rosa}, Christiane Sato, Marcel Silva and Thomas Britz for the examples with $\vv=(6,6,6,7)$ and $\kk=(1,1,1,1)$, and Doug Stinson for the argument in the case $\vv=(6,5,5)$ and $\kk=(2,1,1)$.

\appendix

\section{Exceptional maximum generalized packings for \mbox{$\kk=(3,1)$}} \label{section:31appendix}

In this appendix, we give examples of maximum generalized packings for $\kk=(3,1)$, $t=2$ and $\lambda=1$ which arise as a result of the exceptions in Theorem~\ref{thm:CHW}, for $v_1 = 11$, 12 and 13.

\begin{appexample} \label{example:11pts}
A block colouring of a maximum $2$-$(11,3,1)$ packing with colour type $3^5 1^2$, which yields maximum $2$-$(\vv,\kk,1)$ generalized packings for $\vv=(11,v_2)$, $\kk=(3,1)$ (for all $v_2$):
$$
\begin{array}{l}
\{ \{1,2,3\}, \{4,5,6\}, \{7,8,9\} \}  \\
\{ \{1,4,7\}, \{2,5,8\}, \{3,6,10\} \} \\
\{ \{1,5,9\}, \{2,6,7\}, \{3,8,11\} \} \\
\{ \{1,6,8\}, \{3,4,9\}, \{5,10,11\} \} \\
\{ \{3,5,7\}, \{4,8,10\}, \{6,9,11\} \} \\
\{ \{2,4,11\} \} \\
\{ \{2,9,10\} \}. 
\end{array}
$$
(This example was taken from Colbourn, Horsley and Wang~\cite{colbourn_horsley_wang}.)
\end{appexample}

\begin{appexample} \label{example:20blocks}
A 7-block colouring of a maximum $2$-$(12,3,1)$ packing with colour type $3^6 2^1$, which yields maximum $2$-$(\vv,\kk,1)$ generalized packings for $\vv=(12,v_2)$, $\kk=(3,1)$ where $v_2 \geq 7$:
$$
\begin{array}{l}
\{ \{1, 2, 3\},  \{5, 9, 12\} \} \\
\{ \{1, 4, 5\},  \{2, 11, 12\},  \{3, 9, 10\} \} \\
\{ \{1, 6, 7\},  \{4, 10, 12\},  \{5, 8, 11\} \} \\
\{ \{1, 8, 9\},  \{3, 7, 11\},  \{5, 6, 10\} \}  \\
\{ \{1, 10, 11\},  \{ 2, 4, 6\},  \{7, 8, 12\} \} \\
\{ \{2, 5, 7\},  \{3, 4, 8\},  \{6, 9, 11\} \} \\
\{ \{2, 8, 10\},  \{3, 6, 12\}, \{4, 7, 9\} \}.
\end{array}
$$
(This was found by deleting a point from an 8-block colouring of an $\sts(13)$ with colour type $4^4 3^3 1^1$: cf.\ Example~\ref{example:sts13} below.)
\end{appexample}

\begin{appexample} \label{example:19blocks}
A 5-block colouring of a $2$-$(12,3,1)$ packing with 19 blocks and colour type $3^6 2^1$, which yields maximum $2$-$(\vv,\kk,1)$ generalized packings for $\vv=(12,v_2)$, $\kk=(3,1)$ where $v_2 \leq 6$:
$$
\begin{array}{l}
\{ \{1, 2, 3\}, \{4, 5, 6\}, \{7, 8, 9\} \} \\
\{ \{1, 4, 7\}, \{2, 6, 10\}, \{3, 8, 11\}, \{5, 9, 12\} \} \\
\{ \{1, 5, 8\}, \{2, 9, 11\}, \{3, 6, 7\}, \{4, 10, 12\} \} \\
\{ \{1, 6, 9\}, \{2, 4, 8\}, \{3, 5, 10\}, \{7, 11, 12\} \} \\
\{ \{1, 10, 11\}, \{2, 5, 7\}, \{3, 4, 9\}, \{6, 8, 12\} \}.
\end{array}
$$
(This was found by manipulating a 5-block colouring of 20 blocks on 13 points, obtained by the authors in a computer search.)
\end{appexample}

\begin{appexample} \label{example:sts13}
An 8-block colouring of Steiner triple system $\sts(13)$ with colour type $4^4 3^3 1^1$, which yields maximum $2$-$(\vv,\kk,1)$ generalized packings for $\vv=(13,v_2)$, $\kk=(3,1)$ where $v_2 \geq 7$:
$$
\begin{array}{l}
\{ \{1, 2, 3\},  \{4, 11, 12\},  \{5, 9, 13\} \} \\
\{ \{1, 4, 5\},  \{2, 11, 13\},  \{3, 9, 10\},  \{ 6, 8, 12\} \} \\
\{ \{1, 6, 7\},  \{2, 9, 12\},  \{4, 10, 13\},  \{5, 8, 11\} \} \\
\{ \{1, 8, 9\},  \{3, 7, 11\},  \{5, 6, 10\} \}  \\
\{ \{1, 10, 11\},  \{ 2, 4, 6\},  \{3, 5, 12\},  \{7, 8, 13\} \} \\
\{ \{1, 12, 13\},  \{2, 5, 7\},  \{3, 4, 8\},  \{6, 9, 11\} \} \\
\{ \{2, 8, 10\},  \{3, 6, 13\}, \{4, 7, 9\} \} \\
\{ \{7, 10, 12\} \}.
\end{array}
$$
(This example was found by manipulating $\sts(13)$ number 2 in Mathon, Phelps and Rosa~\cite{mathon_phelps_rosa}.)
\end{appexample}

\begin{appexample} \label{example:13pts}
A 6-block colouring of a $2$-$(13,3,1)$ packing with 24 blocks and colour type $4^6$, which yields maximum $2$-$(\vv,\kk,1)$ generalized packings for $\vv=(13,v_2)$, $\kk=(3,1)$ where $v_2 \leq 6$:
$$
\begin{array}{l}
\{ \{1, 2, 3\}, \{4, 5, 6\}, \{7, 8, 9\}, \{10, 11, 12\} \} \\
\{ \{1, 4, 7\}, \{2, 5, 8\}, \{3, 6, 10\}, \{9, 11, 13\} \} \\
\{ \{1, 5, 9\}, \{2, 10, 13\}, \{3, 7, 11\}, \{6, 8, 12\} \} \\
\{ \{1, 6, 11\}, \{2, 7, 12\}, \{3, 5, 13\}, \{4, 9, 10\} \} \\
\{ \{1, 8, 10\}, \{2, 4, 11\}, \{3, 9, 12\}, \{6, 7, 13\} \} \\
\{ \{1, 12, 13\}, \{2, 6, 9\}, \{3, 4, 8\}, \{5, 7, 10\} \}.
\end{array}
$$
(This example was also taken from Colbourn, Horsley and Wang~\cite{colbourn_horsley_wang}.)
\end{appexample}

\section{Exceptional maximum generalized packings for \mbox{$\kk=(2,1,1)$}} \label{section:211appendix}

In this appendix, we compile a list of generalized packings with $\kk=(2,1,1)$, $t=2$ and $\lambda=1$ in certain small cases.  In particular, these packings arise where the Howell design which we would otherwise use to construct a maximum packing does not exist.

\begin{appexample} \label{ex:433}

{A packing of size 4 where $\vv=(4,3,3)$:}
$$
\begin{array}{|c|c|c|} \hline
1,2 & 3,4 &     \\ \hline
    &     & 1,3 \\ \hline
    &     & 2,4 \\ \hline
\end{array}
$$
\end{appexample}

\begin{appexample} \label{ex:435}

{A packing of size 5 where $\vv=(4,3,5)$:}

$$
\begin{array}{|c|c|c|c|c|} \hline
1,2 & 3,4 &     &     &     \\ \hline
    &     & 1,3 & 2,4 &     \\ \hline
    &     &     &     & 1,4 \\ \hline
\end{array}
$$
\end{appexample}

\begin{appexample} \label{ex:436}

{A packing of size 6 where $\vv=(4,3,6)$:}
$$
\begin{array}{|c|c|c|c|c|c|} \hline
1,2 & 3,4 &     &     &     &     \\ \hline
    &     & 1,3 & 2,4 &     &     \\ \hline
    &     &     &     & 1,4 & 2,3 \\ \hline
\end{array}
$$
\end{appexample}

\begin{appexample} \label{ex:444}

{A packing of size 5 where $\vv=(4,4,4)$:}
$$
\begin{array}{|c|c|c|c|} \hline
1,2 &     &     & 3,4 \\ \hline
    & 1,3 &     &     \\ \hline
    &     & 1,4 &     \\ \hline
    & 2,4 &     &     \\ \hline
\end{array} 
$$
\end{appexample}

\begin{appexample} \label{ex:445}

{A packing of size 6 where $\vv=(4,4,5)$:}
$$
\begin{array}{|c|c|c|c|c|} \hline
1,2 & 3,4 &     &     &     \\ \hline
    &     & 1,3 & 2,4 &     \\ \hline
    &     &     &     & 1,4 \\ \hline
    &     &     &     & 2,3 \\ \hline
\end{array}
$$
\end{appexample}

\begin{appexample} \label{ex:555}

A packing of size 9 where $\vv=(5,5,5)$:
$$
\begin{array}{|c|c|c|c|c|} \hline
1,2 & 3,4 &     &     &  \\ \hline
3,5 &     & 1,4 &     &  \\ \hline
    & 2,5 &     & 1,3 &  \\ \hline
    &     & 2,3 &     & 1,5 \\ \hline
    &     &     & 4,5 &  \\ \hline
\end{array}
$$
\end{appexample}

\begin{appexample} \label{ex:556}

A packing of size 10 where $\vv=(5,5,6)$:
$$
\begin{array}{|c|c|c|c|c|c|} \hline
1,2 & 3,4 &     &  &  &  \\ \hline
3,5 &     & 1,4 &  &  &  \\ \hline
    & 2,5 &  & 1,3 &  &  \\ \hline
&  & 2,3 &  & 1,5 &  \\ \hline
&  &  & 4,5 &  & 2,4 \\ \hline
\end{array}
$$
\end{appexample}

\begin{appexample} \label{ex:622}

A packing of size 4 where $\vv=(6,2,2)$:
$$
\begin{array}{|c|c|} \hline
1,2 & 3,4 \\ \hline
3,5 & 1,6 \\ \hline
\end{array}
$$
\end{appexample}

\begin{appexample} \label{ex:655}

A packing of size 13 where $\vv=(6,5,5)$:
$$
\begin{array}{|c|c|c|c|c|} \hline
1,2 & 3,4 & 5,6 &     &     \\ \hline
3,5 & 1,6 & 2,4 &     &     \\ \hline
4,6 & 2,5 & 1,3 &     &     \\ \hline
    &     &     & 1,4 & 2,3 \\ \hline
    &     &     & 2,6 & 1,5 \\ \hline
\end{array}
$$
\end{appexample}

\begin{appexample} \label{ex:656}

A packing of size 15 where $\vv=(6,5,6)$:
$$
\begin{array}{|c|c|c|c|c|c|} \hline
1,2 & 3,4 & 5,6 &     &     &     \\ \hline
3,5 & 1,6 & 2,4 &     &     &     \\ \hline
4,6 &     &     & 1,5 & 2,3 &     \\ \hline
    & 2,5 &     & 3,6 &     & 1,4 \\ \hline
    &     & 1,3 &     & 4,5 & 2,6 \\ \hline
\end{array}
$$
\end{appexample}

\begin{appexample} \label{ex:666}

{A packing of size 15 where $\vv=(6,6,6)$:}

$$
\begin{array}{|c|c|c|c|c|c|} \hline
1,2 & 3,4 & 5,6 &     &     &     \\ \hline
3,5 & 1,6 & 2,4 &     &     &     \\ \hline
4,6 &     &     & 1,5 & 2,3 &     \\ \hline
    & 2,5 &     & 3,6 &     & 1,4 \\ \hline
    &     & 1,3 &     & 4,5 & 2,6 \\ \hline
    &     &     &     &     &     \\ \hline
\end{array}
$$
\end{appexample}

\begin{appexample} \label{ex:755}

A packing of size 15 where $\vv=(7,5,5)$:
$$
\begin{array}{|c|c|c|c|c|} \hline
1,2 & 3,4 & 5,6 & & \\ \hline
3,5 & 1,6 & 2,4 & & \\ \hline
4,6 & & & 1,3 & 2,5 \\ \hline
& 2,7 & & 4,5 & 3,6 \\ \hline
& & 3,7 & 2,6 & 1,4 \\ \hline
\end{array}
$$
\end{appexample}

\begin{appexample} \label{ex:855}

A packing of size 19 where $\vv=(8,5,5)$:
$$
\begin{array}{|c|c|c|c|c|}\hline
1,2 & 3,4 & 5,6 & 7,8 &     \\ \hline
3,5 & 1,6 & 2,7 &     & 4,8 \\ \hline
4,7 & 2,8 &     & 1,5 & 3,6 \\ \hline
6,8 & 5,7 & 1,4 & 2,3 &     \\ \hline
    &     & 3,8 & 4,6 & 1,7 \\ \hline
\end{array}
$$
\end{appexample}

\begin{appexample} \label{ex:856}

A packing of size 20 where $\vv=(8,5,6)$:
$$
\begin{array}{|c|c|c|c|c|c|}\hline
1,2 & 3,4 & 5,6 & 7,8 &     &     \\ \hline
3,5 & 1,6 & 2,7 &     & 4,8 &     \\ \hline
4,6 & 2,8 &     & 1,5 &     & 3,7 \\ \hline
    & 5,7 & 1,4 &     & 2,3 & 6,8 \\ \hline
    &     & 3,8 & 2,6 & 1,7 & 4,5 \\ \hline
\end{array}
$$
\end{appexample}

\end{document}